\documentclass[12pt]{amsart}

\usepackage[utf8]{inputenc}
\usepackage{geometry}
\usepackage{amsmath, amssymb, amsthm, amsfonts, amsxtra, mathrsfs, mathtools} 
\usepackage{faktor}
\usepackage{multicol}
\usepackage{subcaption}
\usepackage{graphicx}
\usepackage{xcolor}
\usepackage{ragged2e}
\usepackage[intoc]{nomencl}
\makenomenclature

\usepackage{import}
\usepackage{color}
\usepackage{hyperref}
\usepackage[capitalize]{cleveref}
\usepackage{stmaryrd}
\usepackage{tikz-cd}
\usepackage{tikz}
\usepackage{tikzsymbols}
\usepackage{enumitem}
\usepackage[ruled,vlined]{algorithm2e}
\usepackage{autonum}

\usepackage[mathscr]{euscript}

\input xy
\xyoption{all}

\makeatletter
\newtheorem*{rep@theorem}{\rep@title}
\newcommand{\newreptheorem}[2]{%
\newenvironment{rep#1}[1]{%
 \def\rep@title{#2 \ref{##1}}%
 \begin{rep@theorem}}%
 {\end{rep@theorem}}}
\makeatother

\newtheorem{thm}{Theorem}[section]
\newreptheorem{thm}{Theorem}
\newtheorem{lemma}[thm]{Lemma}
\newtheorem{proposition}[thm]{Proposition}
\newtheorem{corollary}[thm]{Corollary}
\newreptheorem{corollary}{Corollary}

\theoremstyle{definition}

\newtheorem{definition}[thm]{Definition}
\newreptheorem{definition}{Definition}
\newtheorem{definitionlemma}[thm]{Definition/Lemma}
\newtheorem{solution}[thm]{Solution}
\newtheorem{example}[thm]{Example}

\newtheorem{remark}[thm]{Remark}

\newtheorem{convention}[thm]{Convention}
    
\def\A{\mathbf{A}}

\def\G{\mathbf{G}}

\def\I{\mathbf{I}}

\def\N{\mathbf{N}}

\def\Q{\mathbb{Q}}
\def\R{\mathbb{R}}

\def\calC{\mathcal{C}}

\def\calM{\mathcal{M}}
\def\calN{\mathcal{N}}

\def\calP{\mathcal{P}}

\def\calS{\mathcal{S}}

\renewcommand{\l}{\lambda}

\newcommand{\s}{\sigma}
\renewcommand{\t}{\tau}

\renewcommand{\div}{\operatorname{div}}

\newcommand\ep{\varepsilon}

\newcommand{\TLM}{\textup{TLM}}

\newcommand{\Mntrop}{M_{0, n}^{\trop}}
\newcommand{\Mwtrop}{M_{0, w}^{\trop}}

\newcommand\psiiwtrop{\psi^{\trop}_{N, w}}

\newcommand\pr{\textup{pr}}

\newcommand\supp{\textup{Supp}}

\newcommand{\trop}{\textup{trop}}

\newcommand{\val}{\operatorname{val}}

\newcommand{\un}[1]{\underline{#1}}

\newlength\mylen
\newlist{mycases}{enumerate}{1}
\setlist[mycases,1]{label=\textbf{Case~\arabic*.}, 
  labelwidth=\dimexpr-\mylen-\labelsep\relax,leftmargin=0pt,align=right}

\newcommand{\repeatable}[2]{%
  \label{#1}\global\@namedef{repeatable@#1}{#2}#2
}
\renewcommand{\repeat}[1]{%
  \@ifundefined{repeatable@#1}{NOT FOUND}{$\@nameuse{repeatable@#1}$}%
  ~\ref{#1}}
\makeatother

\title{Intersection numbers on tropical Hassett spaces}
\author{Marvin Anas Hahn}
\address{Marvin Anas Hahn, Max Planck Institute for Mathematics in the Sciences, 04103 Leipzig, Germany}\email{mhahn@mis.mpg.de}
\author{Shiyue Li}
\address{Shiyue Li, Department of Mathematics, Brown University, Providence, RI 02912, USA}\email{shiyue\_li@brown.edu}


\date{\today}

\begin{document}
\maketitle
\begin{abstract}
    We study the intersection of tropical $\psi$-classes on tropical heavy/light Hassett spaces, generalising a result of Kerber--Markwig for $M^{\trop}_{0, n}$. Our computation reveals that the weight of a maximal cone in an intersection has a combinatorial intepretation in terms of the underlying tropical curve and it is always nonnegative. In particular, our result specialises to that, in top dimension, the tropical intersection product coincides with its classical counterpart.
\end{abstract}

\setcounter{tocdepth}{1}
\tableofcontents

\section{Introduction}
    In this paper, we study intersection products of $\psi$-classes on $\Mwtrop$ where $w$ is heavy/light using tropical intersection theory developed by Allerman and Rau in \cite{Allermann_Rau_2010} as a generalisation of \cite{Kerber_Markwig_2009}. 
    
    Given $g\ge0, n \ge 2$ and $w \in (\Q \cap [0, 1])^n$ such that $2g + \sum w_i > 2$, Hassett \cite{HASSETT2003316} introduced the moduli space $\overline{\calM}_{g, w}$ of  $w$-stable nodal $n$-marked curves of genus $g$ as an alternate compactification of the well-studied moduli space of $n$-marked smooth curves of genus $g$, $\calM_{g, n}$. We work in genus $0$. The moduli space $\overline{\calM}_{0, w}$ parametrises reduced connected rational curves $C$ with $n$ marked points $p_1,\dots,p_n\in C$, such that
    \begin{enumerate}
        \item a collection of points $p_{i_1},\dots,p_{i_s}$ can coincide only if $\sum_{j=1}^sw_{i_j}\le 1$;
        \item the singularities of $C$ are ordinary double points, called \textit{nodes}; and
        \item for any irreducible component $T$ of $C$, 
        \begin{equation}
            \#\{\text{nodes on $T$}\} + \sum_{p_i\in T}w_i> 2.
        \end{equation}
       
    \end{enumerate}
    In particular, when $w$ is the all $1$'s vector, we recover the Deligne--Mumford--Knudsen compactification by stable nodal curves $\overline{\calM}_{g, n}$ of $\calM_{g, n}$. In \cite{Ulirsch_2015}, Ulirsch introduced the tropical analogue of $\overline{\calM}_{g, w}$ parametrising $w$-stable tropical curves of genus $g$, denoted as $M^{\trop}_{g, w}$ and studied its geometry.

    \subsection{Context}
    The family of \textit{$\psi$-classes} represents one of the most studied objects in the intersection theory of moduli spaces of curves. They parametrise curves satisfying certain tangency conditions at the marked points.
    
    While in the 90s the intersection theory of $\psi$-classes on $\overline{\calM}_{g,n}$ was resolved by the Witten--Kontsevich theorem \cite{witten1990two,kontsevich1992intersection}, intersection products of $\psi$-classes on Hassett spaces $\overline{\calM}_{g, w}$ were first studied by Alexeev and Guy in \cite{alexeev_guy_2008}. By the reduction morphism $\rho_w\colon\overline{\mathcal{M}}_{g,n}\to\overline{\mathcal{M}}_{g,w}$ constructed in \cite{HASSETT2003316}, Alexeev and Guy proved that integrals of $\psi$-classes on $\overline{\mathcal{M}}_{g,w}$ can be expressed as linear combinations of $\psi$-classes and certain boundary divisors on $\overline{\mathcal{M}}_{g,n}$. Specifically in genus $0$, \cite{moon2013log,ceyhan2009chow} related Chow classes on $\overline{\mathcal{M}}_{0,n}$ and $\overline{\mathcal{M}}_{0,w}$. Most recently, Blankers and Cavalieri \cite{Blankers_Cavalieri_2020_Wall} extended these results to intersections of $\psi$-classes in arbitrary dimension, which in turn resolved the combinatorial relation between $\psi$-class intersections on $\overline{\mathcal{M}}_{g,n}$ and $\overline{\mathcal{M}}_{g,w}$. 

    
    Tropical geometry provides a combinatorial framework for the intersection theory of $\psi$-classes in genus $0$.  In \cite{zbMATH05543460}, Mikhalkin introduced the moduli space of tropical rational $n$-marked stable curves  $M_{0,n}^{\textrm{trop}}$ as an embedded balanced rational polyhedral fan; there, tropical $\psi$-classes are certain balanced subfans of codimension $1$ corresponding to metric graphs with certain valency conditions. Then Allermann and Rau \cite{Allermann_Rau_2010} delevoped tropical intersection theory on balanced rational polyhedral fans. With this, Kerber and Markwig computed the intersection of tropical $\psi$-classes on $M_{0,n}^{\textrm{trop}}$ in \cite{Kerber_Markwig_2009}; in particular, they showed that the intersection product of tropical $\psi$-classes on $M_{0,n}^{\textrm{trop}}$ recovers their algebro-geometric counterpart on $\overline{\mathcal{M}}_{0,n}$. In the weighted case, Cavalieri, Hampe, Markwig and Ranganathan studied the cone complexes $M^{\trop}_{0, w}$ as tropical compactifications analogous to the work of \cite{gibney2010chow}, and showed that $M_{0,w}^{\trop}$ is a balanced fan if and only if $w$ is \textit{heavy/light} in \cite[Theorem A]{Cavalieri_Hampe_Markwig_Ranganathan_2016}. This allows us to employ the tropical intersection theory developed in \cite{Allermann_Rau_2010}. When $w$ is such a weight vector of length $n$ and with $m$ light weights, the tropicalization $\Sigma_w$ of the torus embedding
    \[
    \calM_{0, w} \hookrightarrow T_{w} := T^{\binom{n}{2}  - \binom{n - m}{2} - m}
    \] has the underlying cone complex $\Mwtrop$. Then the closure of $\calM_{0, w}$ in the toric variety $X(\Sigma_w)$ coincides with Hassett's compactification $\overline{\calM}_{0, w}$ and is indeed a \textit{tropical compactification} in the sense of \cite{tevelev2007compactifications}. In \cite{Kannan_2021}, Kannan, Karp and the second author employed this setup to derive the entire Chow ring $A^{\ast}(\overline{\calM}_{0, w})$, in which the classical weighted $\psi$-classes reside, using toric intersection theory on $X(\Sigma_w)$.
    
    \subsection{Main results} 
    Our main result is an explicit formula for the intersection products of tropical $\psi$-classes on $\Mwtrop$, in the case when $w$ is heavy/light. Our formula shows that the weight of each maiximal cone appearing in the intersection product has a combinatorial description given by the underlying tropical curve. See \cref{subsec:psi-classes-as-pushforward}, \cref{subsec:tropical-intersection} for weighted tropical $\psi$-classes and backgrounds on $\Mwtrop$. 
    
    \begin{thm}
    \label{thm:main-later}
    Let $n \ge 4$, $n - m \ge 2$, $0 < \ep < 1/m$, $w = (1^{(n-m)}, \ep^{(m)})$, and ${K = (k_i)_{i \in [n]} \in (\mathbb{Z}_{\ge 0})^n}$. The intersection product $\prod\limits_{i = 1}^{n} \psi_{i, w}^{k_i}$ is the weighted subfan of $\Mwtrop$ consisting of closures of the cones of codimension $\sum_i k_i$ satisfying the following conditions:
    \begin{enumerate}
        \item For each maximal cone $\sigma$ in $\prod\limits_{i = 1}^{n} \psi_{i, w}^{k_i}$ with combinatorial type $\G_{\sigma} = (G_{\sigma}, m_{\sigma})$, and for each vertex $v \in V(G_{\sigma})$, we have
        \[
        \val(v) + |m_{\sigma}^{-1}(v)| = 3 + \sum_{i \in m_{\sigma}^{-1}(v)} k_i. 
        \]
        \item The weight of a maximal cone $\sigma$ in $\prod\limits_{i = 1}^{n} \psi_{i, w}^{k_i}$ with combinatorial type $\G_{\sigma} =(G_{\sigma}, m_{\sigma})$ is the product of the tropical local multiplicities at all vertices of $G_{\sigma}$, i.e. 
        \begin{align}
        \label{eq:tau-weight}
        \omega(\sigma) &= \prod_{v \in V(G_{\sigma})} \TLM_{\sigma}(v)\\
        &= \prod_{v \in V(G_{\sigma})} \sum_{P \in \calP_{w}(v)} (-1)^{|m_{\sigma}^{-1}(v)| - \ell(P)} \binom{\sum_{i} K(P)_i}{K(P)_1,\cdots, K(P)_{\ell(P)}}.
    \end{align}
    \end{enumerate}
    \end{thm}
    
    From this result, we obtain two immediate corollaries. Firstly, when $w=(1^{(n)})$, we recover the characterisation of intersection products of tropical $\psi$-classes on $M_{0,n}^\trop$ derived by Kerber and Markwig in \cite{Kerber_Markwig_2009}.
    
    \begin{corollary}[{\cite[Theorem 4.1]{Kerber_Markwig_2009}}]
    \label{cor-kerbmark}
    Let $w = (1^{(n)})$ and $K = (k_1, \ldots, k_n) \in (\mathbb{Z}_{\ge 0})^n$. The intersection product $\prod_{i} \psi_i^{k_i}$ is the weighted subfan of $\Mntrop$ consisting of closures of the cones of codimension $\sum_{i}k_i$ satsfying the following conditions:
    \begin{enumerate}
        \item For each maximal cone $\sigma$ in $\prod_{i} \psi_i^{k_i}$ with combinatorial type $\G_{\s} = (G_{\sigma}, m_{\sigma})$ and for each vertex $v \in V(G_{\sigma})$, 
        \[
        \val(v) + |m_{\sigma}^{-1}(v)| = 3 + \sum_{i \in m_{\sigma}^{-1}(v)} k_i. 
        \]
        \item The weight of a maximal cone $\sigma$ in $\prod_{i} \psi_i^{k_i}$ with combinatorial type $\G_{\s} = (G_{\sigma}, m_{\sigma})$ is
        \[
        \omega(\sigma) = \prod_{v \in V(G_{\sigma})}\binom{\sum_{i \in m_{\sigma}^{-1}(v)} k_i}{k_1,\cdots, k_{|m_{\sigma}^{-1}(v)|}}. 
        \]
    \end{enumerate}
    \end{corollary}

    The work of Katz in \cite{Katz_2012} draws connections between toric and tropical intersection theories and naturally leads to the expectation that the degrees of the top-dimensional intersection products of $\psi$-classes, i.e. when $\sum_{i} k_i = n - 3$, on $M_{0,w}^\trop$ coincide with their algebro-geometric counterparts. The following corollary confirms this expectation, in the sense that we indeed recover the intersection product of $\psi$-classes of Hassett spaces computed in \cite[Theorem 7.9]{alexeev_guy_2008}.

    \begin{corollary}
    \label{cor-tropalg}
    Let $w$ heavy/light. When $\sum k_i=n-3$, the intersection product $\prod\psi_i^{k_i}$ is of dimension $0$, consists of precisely one cone $\{0\}$ and the weight of the cone $\{0\}$ of $M_{0,w}^\trop$ is
    \begin{equation}
    \sum_{P\in \calP_w([n])}(-1)^{n-\ell(P)}\binom{\ell(P)-3}{K(P)_1, \cdots, K(P)_{\ell(P)}}.
    \end{equation}
    \end{corollary}

    \begin{remark}
    \label{rem:compare-katz}
        Note that it is possible to obtain \cref{cor-tropalg} by generalising \cite[Proposition 7.5]{Katz_2012} to the heavy/light weighted case. However, the application of the fan displacement rule as one lifts and intersects the Chow classes in a toric variety following \cite[Theorem 6.3]{Katz_2012} is more complicated than a direct computation via tropical intersection theory in the present paper due to less symmetry of the weight vector $w$. Moreover, our approach features the advantage of a combinatorial description of intersection products of tropical $\psi$-classes in any dimension.
    \end{remark}
    
    The starting point of the proof of \cref{thm:main-later} is a combinatorial characterisation of weighted tropical $\psi$-classes on $\Mwtrop$ in \cref{thm:psi-class-characterisation}. In \cref{subsec:psi-classes-as-pushforward}, we define tropical $\psi$-classes on $\Mwtrop$ via pushforward of linear combinations of $\psi$-classes and boundary divisors on $M_{0,n}^\trop$ along the projection morphism
    \[
    \pr_{w}^{\trop}: M_{0,n}^{\trop}\to M_{0,w}^{\trop}, 
    \] defined by contracting cones in $\Mntrop$ parametrising $w$-unstable tropical curves. \cref{thm:psi-class-characterisation} below shows that their combinatorial descriptions depend solely on the weight vector $w$. See \cref{subsec:tropical-curves} for terminology on the moduli space $\Mwtrop$. 
    \begin{thm}
    \label{thm:psi-class-characterisation}
        Let $n \ge 4$, $n - m \ge 2$, $0 < \ep < 1/m$, $w = (1^{(n-m)}, \ep^{(m)})$.  
        For $N \in [n]$, we have the following two cases. 
        \begin{enumerate}
            \item if $N$ is heavy, then the class $\psi_{N, w}$ is the balanced subfan of $\Mwtrop$ that is the union of closed cones $\sigma$ of dimension $n-4$ with associated combinatorial type $\G_{\sigma} = (G_{\sigma}, m_{\sigma})$ such that $G_{\sigma}$ has a unique vertex $v$ satisfying
            \[\val(v) + |m_{\sigma}^{-1}(v)| = 4 \text{ and } N \in m_{\sigma}^{-1}(v).\] 
            \item if $N$ is light, then $\psi_{N, w}$ is the balanced subfan of $\Mwtrop$ that is the union of closed cones $\sigma$ of dimension $n-4$ with associated combinatorial type $\G_{\sigma} = (G_{\sigma}, m_{\sigma})$ such that 
            \begin{enumerate}
                \item $G_{\sigma}$ has a unique vertex $v$ such that 
                \[\val(v) + |m_{\sigma}^{-1}(v)| = 4 \text{ and } N \in m_{\sigma}^{-1}(v).\] 
                \item $\sigma$ is not contained in maximal cones contracted by $\pr_{w}^{\trop}: \Mntrop \to \Mwtrop$. 
            \end{enumerate}
        \end{enumerate}
    \end{thm}

    Another key in the proof of \cref{thm:main-later} is a realisability result for the weighted tropical $\psi$-classes. We show in \cref{thm:divisor-of-rat-multiple-of-psi} below that the weighted tropical $\psi$-classes are realised as rational multiples of the tropical Weil divisors of a family of rational functions on $\Mwtrop$, analogous to the case of $\Mntrop$ studied in {\cite[Proposition 3.5]{Kerber_Markwig_2009}}. See \cref{subsec-rationalfunc} for terminology on rational functions on balanced polyhedral fans; see \cref{subsec:tropical-intersection} for backgrounds on tropical boundary divisors and tropical Weil divisors. 
    \begin{thm}
    \label{thm:divisor-of-rat-multiple-of-psi}
    For $n \ge 4$, $n - m \ge 2$ and $0< \varepsilon < 1/m$, $w=(1^{(n-m)},\varepsilon^{(m)})$, and $N \in [n]$, we have the following equality of tropical divisors
    \[\div(f_{N,w}) = K(N, w) \psi_{N, w},\] 
    where the coefficients $K(N, w)$ depend on $n, m$ as follows. 
    \begin{enumerate}
        \item if $n - m =2$, i.e. the weight vector $w$ contains exactly $2$ heavy weights, then 
        \[
        K(N, w) = \begin{cases} 
        m& \text{if } N \in [n-m],\\
        2m - 2  & \text{otherwise}; 
        \end{cases}
        \]
        \item if $n - m > 2$, then 
        \[
        K(N, w) = \begin{cases} 
        \binom{n - 1}{2} - \binom{m}{2} & \text{if } N \in [n-m],\\
        \binom{n-1}{2} - \binom{m-1}{2}  & \text{otherwise}. 
        \end{cases}
        \]
    \end{enumerate}
    \end{thm}

\subsection{Future directions}        
    In this paper, we have focused on the combinatorics of intersection products of heavy/light weighted tropical $\psi$-classes in genus $0$. We aim to point to two related works that provide interesting future directions. Firstly, Fry extended \cite{CHMR2014moduli} to the moduli spaces of rational tropical curves with stability conditions given by a graph and an ordered partition on the vertices in \cite{fry2019tropical}. When the graph is complete multipartite, Fry \cite[Theorem 3.28]{fry2019tropical} identifies the moduli space as the Bergman fan of the graphic matroid, which is balanced. A key assumption for our work is that the moduli space of rational stable tropical curves associated to a heavy/light weight is \textit{balanced}, proved in \cite[Theorem I]{CHMR2014moduli}. It would be interesting to extend our methods to this new family of tropical moduli spaces.
    Secondly, Cavalieri, Gross and Markwig developed a theory of tropical $\psi$-classes on stable elliptic curves in \cite{cavalieri2020tropical}. In particular, the authors derive a correspondence theorem and recover the $1$-point elliptic $\psi$-integral on the tropical side \cite[Theorem A]{cavalieri2020tropical}. In light of the present paper, a generalisation of their work to the weighted case is an interesting topic of further research.
    
    \subsection{Acknowledgement} We thank Hannah Markwig and Renzo Cavalieri for several helpful discussions during the work on this manuscript. Furthermore, we thank two anonymous referees for their careful reading and insightful comments. The first author gratefully acknowledges financial support by the Max Planck Gesellschaft. The second author would like to express thanks to Bernd Sturmfels and Max Planck
    Institute for Mathematics in the Sciences for their hospitality during this project. 

\section{Moduli spaces of tropical weighted stable curves}
\label{sec:tropical-hassett-spaces}
In this section, we review the basics about $\Mwtrop$ in \cref{subsec:tropical-curves} and recall backgrounds on tropical intersection theory in \cref{subsec:tropical-intersection}. 

\subsection{Tropical $w$-stable curves and their moduli}
\label{subsec:tropical-curves}
Let $n \ge 2$, and let $w \in (\Q \cap (0, 1])^n$ be a weight vector. We start from scratch with the following definitions. 
\begin{definition}
    A \textbf{rational $n$-marked graph} is a tuple $\G = (G, m)$ such that 
    \begin{enumerate}
        \item $G$ is a finite tree with vertex set $V(G)$ and edge set $E(G)$; 
        \item $m\colon [n] \to V(G)$ is a function (called the ``marking function" of $G$).
    \end{enumerate}
\end{definition}

\begin{definition}
    A rational $n$-marked graph $\G = (G,m)$ is \textbf{$w$-stable} if for all $v \in V(G)$
    \[
        \val(v) + \sum_{i \in m^{-1}(v)} w_i > 2. 
    \]
\end{definition}

\begin{definition}
    An \textbf{abstract rational tropical $w$-stable curve} is a tuple $(\mathbf{G}, \ell)$ where $\mathbf{G}$ is a rational $w$-stable graph and $\ell\colon E(G)\to\R_{\ge0}$ is a function. 
\end{definition}
We call $\mathbf{G}$ the \textit{combinatorial type} and $\ell$ the \textit{length function} of an abstract rational tropical $w$-stable curve.

The moduli space of abstract rational tropical $w$-stable curves $M_{0, w}^{\trop}$ parametrises all abstract rational tropical $w$-stable curves. When $w=(1^{(n)})$ is the all $1$'s vector, we recover the moduli space of tropical $n$-marked curves $M^{\trop}_{0, n}$. In this paper, we focus on \textit{heavy/light} weight vectors, recalled as follows. 

\begin{definition}
Let $w = (w_1,\dots,w_n)\in \mathbb{Q}^n \cap (0,1]^n$ be a vector of weights. 
\begin{enumerate}
    \item We call $i\in\{1,\dots,n\}$ \textit{heavy in $w$} if for all $j\neq i$ we have $w_i+w_j>1$.
    \item We call $i\in\{1,\dots,n\}$ \textit{small in $w$} if $w_i+w_j>1$ implies $j$ is heavy in $w$. 
\end{enumerate}
If in addition the total weight of all small weights is less than $1$, we call them \textit{light}. 
\end{definition} 

For example, the weight vector $w = (1, 1, 3/4, 1/2)$ has all $w_i$ heavy, whereas the weight vector $w' = (1, 1, 1/3, 1/3)$ has $w'_1, w'_2$ heavy, and $w'_3, w'_4$ light. 

\begin{convention}
\label{conv:weight-vector}
    We hereafter assume that $w$ is heavy/light unless specified otherwise. As it is customary, we order the weights of a heavy/light weight vector such that all heavy weights precede light weights. Explicitly, we always take $n \ge 4$, $n - m \ge 2$, $0 < \varepsilon < 1/m$ and let $w = (1^{(n-m)}, \varepsilon^{(m)})$.
\end{convention}  

We now describe an embedding of $M_{0,w}^{\trop}$. Let $(\G, \ell)$ be an abstract rational tropical $w$-stable curve, and let $i,j\in[n]$. We define the distance function $\mathrm{dist}_{(\G, \ell)}: [n]^2 \to \R_{\ge 0}$ as follows: 
\begin{equation}
    \mathrm{dist}_{(\G, \ell)}(i, j) :=\min\limits_{P \in P(i, j)}\bigg(\sum_{e\in P}\ell(e) \bigg),
\end{equation}
where the minimum is taken over all paths $P(i, j)$ in $(\G, \ell)$ from the vertex supporting $i$ to the vertex supporting $j$. Moreover, we define the map 
\begin{align}
    \phi_{w}\colon\mathbb{R}^n&\to\mathbb{R}^{\binom{n}{2}-\binom{m}{2}}\\
    (a_1,\dots,a_n) &\mapsto(a_i+a_j), 
\end{align} for all $\{i, j\} \not\subseteq [n] \smallsetminus [n-m]$. 

Then, we consider the map
\begin{align}
    \Phi_{w}: M_{0,w}^\trop&\to \R_w := \faktor{\mathbb{R}^{\binom{n}{2}-\binom{m}{2}}}{\mathrm{Im}(\phi_w)}
\end{align}
defined by 
\begin{align}
    (\G, \ell) &\mapsto(\mathrm{dist}_{(\G, \ell)}(i,j)). 
\end{align} for all $\{i,j\} \not\subseteq [n] \smallsetminus [n-m]$. 

A similar argument in \cite[Example 7.2]{franccois2013diagonal} shows that this map embeds $M_{0,w}^\trop$ as a fan into $\R_w$, whose cones are in bijection with the combinatorial types of abstract rational tropical $w$-stable curves. In particular, a top-dimensional cone $\sigma$ with combinatorial types $\G_{\sigma} = (G_{\sigma}, m_{\sigma})$ satisfies that for every vertex $v \in V(G_{\sigma})$, $\mathrm{val}(v)+|m^{-1}(v)|=3$. 

\begin{example}
The space $M_{0, 5}^{\trop}$ is the cone over the Petersen graph (\cref{fig:tropical}, left). Its image under the contraction map $\pr_{w}^{\trop}$ for $w = (1^{(2)}, \varepsilon^{(3)})$ is the tropical Losev-Manin space $M_{0, w}^{\trop}$, which is the cone over the $1$-skeleton of the permutohedron associated with $S_3$ (\cref{fig:tropical}, right). 

\begin{figure}[ht]
    \centering
\definecolor{ao(english)}{rgb}{1, 0.4, 0.7}

\begin{tikzpicture}[style=thick]
\draw (18:2cm) -- (90:2cm) -- (162:2cm) -- (234:2cm) --
(306:2cm) -- cycle;
\draw (18:1cm) -- (162:1cm) -- (306:1cm) -- (90:1cm) --
(234:1cm) -- cycle;
\foreach \x in {18,90,162,234,306}{
\draw (\x:1cm) -- (\x:2cm);
\filldraw [](\x:2cm) circle (2pt);
\filldraw [](\x:1cm) circle (2pt);
}
\node[inner sep=1.5pt,circle,ao(english),draw,fill,label={[above,scale=0.5]:{$(14, 235)$}}] at (90:2cm) {};
\filldraw [ao(english)](90:2cm) circle (2pt);
\node[inner sep=1.5pt,circle,ao(english),draw,fill,label={[xshift=-0.2cm, scale=0.5]:{$(25, 134)$}}] at (162:2cm) {};
\filldraw [ao(english)](162:2cm) circle (2pt);
\node[inner sep=1.5pt,circle,ao(english),draw,fill,label={[xshift=-0.5cm, yshift=-0.5cm, scale=0.5]:{$(34, 125)$}}] at (234:2cm) {};

\node[inner sep=1.5pt,circle,ao(english),draw,fill,label={[xshift=-0.4cm, yshift=0.0cm, scale=0.5]:{$(15, 234)$}}] at (234:1cm) {};
\filldraw [ao(english)](234:2cm) circle (2pt);
\node[inner sep=1.5pt,circle,ao(english),draw,fill,label={[xshift=0.0cm, yshift=-0.7cm, scale=0.5]:{$(24, 135)$}}] at (18:1cm) {};
\filldraw [ao(english)](18:1cm) circle (2pt);
\node[inner sep=1.5pt,circle,ao(english),draw,fill,label={[xshift=0.2cm, scale=0.5]:{$(35, 124)$}}] at (18:2cm) {};
\node[inner sep=1.5pt,circle,draw,label={[xshift=0.4cm, yshift=0.0cm,scale=0.5]:{$(23, 145)$}}] at (90:1cm) {};
\node[inner sep=1.5pt,circle,draw,label={[xshift=0.2cm, yshift=0.0cm,scale=0.5]:{$(13, 245)$}}] at (162:1cm) {};
\node[inner sep=1.5pt,circle,draw,label={[xshift=-0.3cm, yshift=-0.5cm,scale=0.5]:{$(12, 345)$}}] at (306:1cm) {};
\node[inner sep=1.5pt,circle,draw,label={[xshift=0.3cm, yshift=-0.5cm,scale=0.5]:{$(45, 123)$}}] at (306:2cm) {};

\begin{scope}[shift={(7,0)}]
   \newdimen\R
   \R=1.7cm
   \draw (0:\R) \foreach \x in {60,120,...,360} {  -- (\x:\R) };
   \foreach \x/\l/\p in
     { 60/{$(14, 235)$}/above,
      120/{$(25, 134)$}/above,
      180/{$(34, 125)$}/left,
      240/{$(15, 234)$}/below,
      300/{$(24, 135)$}/below,
      360/{$(35, 124)$}/right
     }
     \node[inner sep=1.5pt,circle,ao(english),draw,fill,label={\p, scale=0.6:\l}] at (\x:\R) {};
\end{scope}

\end{tikzpicture}
    \caption{Each label $(ij, k\ell m)$ indicates the tropical curve with $2$ vertices connected by $1$ edge with the left vertex supporting the marks $i, j$, and the right vertex supporting the marks $k, \ell, m$. In $M_{0, n}^{\trop}$, the $1$-dimensional cones, indicated as black vertices, are contracted under $\pr_w^{\trop}$.}
    \label{fig:tropical}
\end{figure}
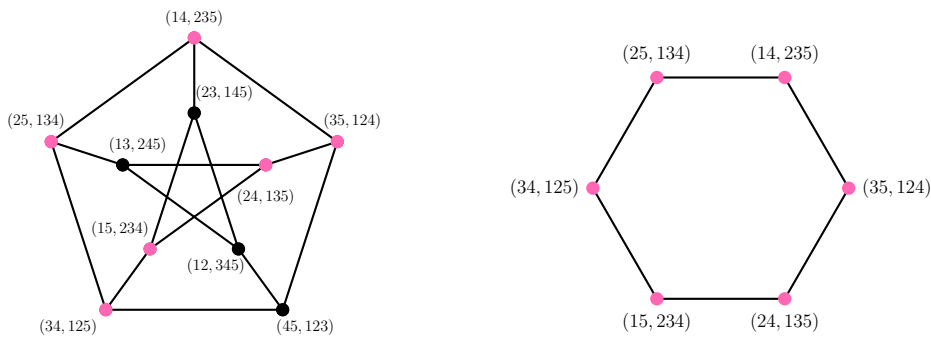
\end{example}

The reason for focusing on $M_{0,w}^{\trop}$ for heavy/light $w$ is the following theorem, and the fact that tropical intersection theory requires that the fans in question be balanced.
\begin{thm}[{\cite[Theorem I]{Cavalieri_Hampe_Markwig_Ranganathan_2016}}]
    The morphism of rational polyhedral fans
    \[
    \pr_w^{\trop}: M_{0, n}^{\trop} \to M_{0, w}^{\trop}
    \]  contracts all cones parametrising $w$-unstable tropical curves. Its image is a balanced fan if and only if the weight $w$ is heavy/light.
\end{thm}

We note that with respect to the embedding we constructed, the map $\mathrm{pr}_w^{\trop}$ corresponds to a projection $\mathbb{R}^{\binom{n}{2}}\to\mathbb{R}^{\binom{n}{2}-\binom{m}{2}}$ which forgets the coordinates indexed by $\{i,j\}\subset[n]\smallsetminus[n-m]$.

\subsection{Tropical Weil divisors and boundary divisors}
\label{subsec:tropical-intersection} 
In this section, we recall the definitions of tropical Weil divisors and boundary divisors from \cite{rau2016intersections}. Let $X$ be a rational, weighted, pure-dimensional polyhedral fan in $\mathbb{R}^n$. For each cone $\tau$ and each cone $\sigma$ containing $\tau$ such that ${\dim \sigma = \dim \tau + 1}$, there is a primitive integral vector $u_{\sigma/\tau}$ in $\mathbb{Z}^n$ such that 
\[
\mathbb{Z} u_{\sigma/\tau} + \tau \cap \mathbb{Z}^n = \sigma \cap \mathbb{Z}^n, 
\] called \textbf{the primitive generator of $\sigma$ with respect to $\tau$}. Writing the weight of each cone $\sigma$ of $X$ by $w(\sigma)$, we say that $X$ is \textbf{balanced} if for any codimension-$1$ cone $\tau$ in $X$, 
\[
\sum_{\substack{\sigma \supsetneq \tau,\\
\dim \sigma = \dim X}} w(\sigma) u_{\sigma/\tau} \in \tau.  
\]
If $X$ is balanced and of dimension $d$, we call $X$ a \textbf{$d$-cycle}. Let $|X|$ denote the support of $X$, i.e. $|X|$ is the union of all cones in $X$. 
\begin{definition}
A \textbf{nonzero rational function on a $d$-cycle $X$} is a nonzero continuous piecewise linear function $\varphi\colon |X|\to\mathbb{R}$ that is linear with a rational slope on each cone.
\end{definition}

\begin{definition}
The \textbf{tropical Weil divisor} $\div(\varphi)$ of a nonzero rational function $\varphi$ on a $d$-cycle $X$ is the weighted codimension-$1$ skeleton of $X$ with the weight for each codimension-$1$ cone $\tau$ given as follows:
\begin{equation}
    \omega_{\varphi}(\tau)=\sum\limits_{\substack{\sigma \supsetneq \tau\\ \dim \sigma := \dim \tau + 1}}\varphi(\omega_{\varphi}(\sigma)u_{\sigma/\tau})-\varphi\bigg(\sum\limits_{\substack{\sigma \supsetneq \tau\\ \dim \sigma = \dim \tau + 1}}\omega_{\varphi}(\sigma)u_{\sigma/\tau}\bigg). 
\end{equation}
\end{definition}
Let $d' < d$ and $Z$ be a $d'$-cycle in $X$ and $\varphi$ a nonzero rational function on $X$. Then, the tropical intersection product of $Z$ and $\mathrm{div}(\varphi)$ is the Weil divisor of $\varphi_{|Z}$. 

By the work of \cite{zbMATH05543460} and \cite{Kerber_Markwig_2009}, $M_{0,n}^{\trop}$ can be embedded as a rational weighted balanced polyhedral fan in a real vector space. We recall tropical boundary divisors on $\Mntrop$ as follows. To begin with, for $I \subsetneq [n]$, we denote by $v_I$ the primitive generator of the cone in $M_{0,n}^{\trop}$ that has the combinatorial type of one bounded edge, and two vertices supporting marks in $I$ and marks in $I^{c}$ respectively. The tropical boundary divisor indexed by $I$ can be defined as the tropical Weil divisor associated with the following rational function given by $I$ on $\Mntrop$. 
\begin{definition}
For $I\subsetneq [n]$, the rational function $\varphi_{I}$ on $M_{0,n}^\trop$ is the linear extension of the map defined as follow: for each primitive generator $v_{S}$,
\begin{equation}
    \varphi_I(v_{S}) :=\begin{cases}
    1\quad\textrm{if}\,I=S\,\textrm{or}\,I^c=S,\\
    0\quad\textrm{otherwise}.
    \end{cases}
\end{equation}
The \textbf{tropical boundary divisor} $D_I^{\trop}$ is defined to be the tropical Weil devisior of $\varphi_I$, i.e. 
\begin{equation}
D_I^{\trop} :=\mathrm{div}(\varphi_I).
\end{equation}
\end{definition}
We denote $D^{\trop}_{I}$ by $D_{I}$ for simplicity. In \cite[Lemma 2.5]{rau2016intersections}, Rau expresses each tropical divisor $D_{I}$ for $I \subsetneq [n]$ as an integral linear combination of codimension-1 cones as follows. Each codimension-1 cone $\sigma$ in $\Mntrop$ corresponds to a combinatorial type $\G_{\sigma} = (G_{\sigma}, m_{\sigma})$ possessing a unique vertex $v$ satisfying $\mathrm{val}(v)+|m_{\sigma}^{-1}(v)|=4$.  Upon denoting the elements in the edge set of $v$ unioned with $m_{\sigma}^{-1}(v)$ as $e_1, e_2, e_3, e_4$, we obtain a partition
\[
P(\sigma) = P_1(\sigma) \sqcup P_2(\sigma) \sqcup P_3(\sigma) \sqcup P_4(\sigma) \vdash [n], 
\] where each part 
\[
P_i(\sigma) = \begin{cases}
\{e_i\} & \text{if } e_i \in m_{\sigma}^{-1}(v), \\
\{\text{marks supported on the component of $G_{\sigma} \smallsetminus \{v\}$ that contains $e_i$\}} & \text{otherwise}.
\end{cases}
\]

\begin{figure}[h!]
    \begin{center}
    \begin{tikzpicture}[scale=0.9]
    \node (A) at (-0.75,0.75) {$P_1(\sigma)$};
    \node (B) at (-0.75,-0.75) {$P_2(\sigma)$};
    \node (C) at (0.75,0.75) {$P_3(\sigma)$};
    \node (D) at (0.75,-0.75) {$P_4(\sigma)$};
    \path (A) edge (D);
    \path (B) edge (C);
    \draw node[fill,circle,scale=0.3]{} (0,0);
    \node (V) at (0, 0.25) {$v$};
    \end{tikzpicture}
    \end{center}
    \caption{The neighborhood of the unique vertex $v$ in $G_{\sigma}$ satisfying ${\mathrm{val}(v)+|m^{-1}(v)|=4}$ for a codimension-$1$ cone $\sigma$ in $\Mntrop$. Here we visualise $m_{\sigma}^{-1}(v)$ as half-edges. }
    \label{fig:combo-type-codim-1}
\end{figure}
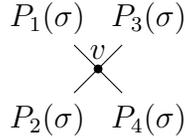

\begin{lemma}[{\cite[Lemma 2.5]{rau2016intersections}}]
\label{lem:boundary-divisor-weight}
    Let $I \subseteq [n]$. The tropical divisor $D_I$ is an integral linear combination of codimension-$1$ cones of $\Mntrop$. More precisely, \sloppy
    \[
    D_I = \sum\limits_{\dim(\sigma) = \dim(X) - 1} c(\sigma) \sigma,
    \] where the coefficient $c(\sigma)$ is 
    \begin{equation}
    \omega(\sigma):=\begin{cases}
    1\, &\textrm{if}\,I=P_i(\sigma)\cup P_j(\sigma),\, \{i, j\} \subset [4],\\
    -1\, &\textrm{if}\, I=P_i(\sigma)\textrm{ or }\, I^c=P_i(\sigma),\,  i \in [4],\\
    0\, &\textrm{otherwise}. 
    \end{cases}
    \end{equation}
\end{lemma}

\begin{example}
\label{exmp:boundary-divisor}
    We describe $D_{45}$ on $M_{0, 5}^{\trop}$. The relevant combinatorial types are shown in Figure \ref{fig:boundary-divisor-12}, where the first three combinatorial types have weight $1$ and the last one has weight $-1$ in $D_{45}$. Therefore, the tropical boundary divisor $D_{45}$ can be written as linear combination of codimension-1 cones as follows. 
    \begin{align}
    D_{45} &= \text{$1$-dimensional cones with generators } \{v_{12}, v_{13}, v_{23}\} \\
    &- \text{$1$-dimensional cone with generator } \{v_{45}\}. 
    \end{align}
    \begin{figure}[h!]
        \centering
        \includegraphics[scale=0.8]{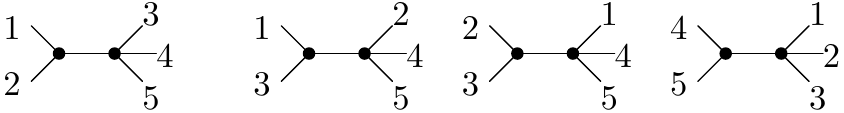}
        \caption{The four combinatorial types appear in the boundary divisor $D_{45}$, where the first three appear with weight $1$ and the last appears with weight $-1$.}
        \label{fig:boundary-divisor-12}
    \end{figure}
\end{example}

\section{$\psi$-classes on $\Mwtrop$}
\label{sec-weildiv}
In this section, we introduce weighted tropical $\psi$-classes on $M_{0, w}^{\trop}$. In \cref{subsec:psi-classes-as-pushforward}, we define weighted tropical $\psi$-classes as pushforwards of certain tropical cycles along $\pr_w^{\trop}$. In \cref{sec-cones}, we give a purely combinatorial description of the cones in these weighted tropical $\psi$-classes, which enables us to show in \cref{subsec-rationalfunc} that each weighted tropical $\psi$-class is a rational multiple of the tropical Weil divisor of a rational function, generalising the result for $w=(1^{(n)})$ in \cite{Kerber_Markwig_2009}.

\subsection{$\psi$-classes on $\Mwtrop$ as pushforward of tropical cycles}
\label{subsec:psi-classes-as-pushforward}\sloppy
Given a weight vector $w$, we define an abstract simplicial complex
\[
\calC_w := \left\{S \subseteq [n]: \sum_{i \in S} w_i \le 1\right\}. 
\] Let $\calC^2_{w}$ denote the elements of $\calC_w$ of cardinality at least $2$.  For example, for
$0 < \varepsilon < 1/3$ and $w = (1^{(2)}, \varepsilon^{(3)})$, we have $\calC_w = \{1, 2, 3, 4, 5, 34, 35, 45, 345\}$ and $\calC_{w}^2 = \{34, 35, 45, 345\}$.

Firstly, we recall tropical $\psi$-classes on $M_{0,n}^\trop$. 
\begin{definition}[{\cite[Definition 3.1]{zbMATH05543460}}]
\label{def-psimikh}
For $N \in [n]$, the tropical $\psi$-class $\psi_N$ is the weighted balanced fan of the closed cones of dimension $n-4$ in $\Mntrop$ such that for each maximal cone $\sigma$ in $\psi_N$ with corresponding combinatorial type $\G_{\sigma} = (G_{\sigma}, m_{\sigma})$, $G_{\sigma}$ have a unique vertex $v$ satisfying \[\mathrm{val}(v)+|m^{-1}(v)|=4 \text{ and } N \in m_{\sigma}^{-1}(v).\] The weight of each cone in $\psi_{N}$ is $1$. 
\end{definition}

We are now ready to define \textit{weighted} tropical $\psi$-classes on $\Mwtrop$.
\begin{definition}
\label{def-psiw}
For $w$ as in \cref{conv:weight-vector} and $N \in [n]$, the \textbf{weighted tropical $\psi$-class $\psi_{N, w}$} is defined as
\begin{equation}
    \psiiwtrop:=\left(\mathrm{pr}_w^{\trop}\right)_\ast (\psi_N^{\trop}-\sum_{N \in S \in \calC_w^2} D_{S}^{\trop}).
\end{equation}
\end{definition}

\begin{remark}
    This definition of weighted tropical $\psi$-classes is inspired by the classical $\psi$-classes on $\overline{\calM}_{g, w}$, studied first by Alexeev and Guy in \cite{alexeev_guy_2008}. The authors proved that ${\psi_N=\pi_w^\ast(\psi_{N,w})+\sum_{S \in \calC^2_{w}} D_{S}}$, where the divisor $D_{S}$ parametrises all rational curves with a ``rational-tail" component marked by $S \in \calC^{2}_w$ and $\pi_{w}$ is the projection morphism $\overline{\calM}_{g, n} \to \overline{\calM}_{g, w}$. 
\end{remark}

\begin{example}
\label{exmp:111ee-psi-4}
    Let $w = (1^{(3)}, \varepsilon^{(2)})$. Firstly, we compute $\psiiwtrop$ using \cref{def-psiw} as follows.
    \begin{align}
        \psi_{4, w}^{\trop} &= \pr_w^{\trop}(\psi_4^{\trop} - D_{45}) =  \pr^{\trop}_w(\psi_4^{\trop}) - \pr^{\trop}_w(D_{45}).
    \end{align}
    By \cref{exmp:boundary-divisor}, this becomes
    \begin{align}
    &\pr^{\trop}_w(\text{$1$-dimensional cones with generators } \{v_{12}, v_{13}, v_{15}, v_{23}, v_{25}, v_{35}\})\\
    &- \pr^{\trop}_{w}(\text{$1$-dimensional cones with generators } \{v_{12}, v_{13}, v_{23}\}). 
    \end{align}
    Thus $\psi_{4, w}^{\trop}= \text{$1$-dimensional cones with generators } \{v_{15}, v_{25}, v_{35}\}.$
\end{example}

\subsection{A combinatorial description of $\psi$-classes}
\label{sec-cones} 
We give a completely combinatorial description of tropical weighted $\psi$-classes on $\Mwtrop$, which will come handy when computing their intersection products. This description is analogous to \cite[Definition 3.1]{zbMATH05543460} and {\cite[Definition 3.1]{Kerber_Markwig_2009}. 
\begin{repthm}{thm:psi-class-characterisation}
        Let $n \ge 4$, $n - m \ge 2$, $0 < \ep < 1/m$, $w = (1^{(n-m)}, \ep^{(m)})$.  
        For $N \in [n]$, we have the following two cases. 
        \begin{enumerate}
            \item if $N$ is heavy, then the class $\psi_{N, w}$ is the balanced subfan of $\Mwtrop$ that is the union of closed cones $\sigma$ of dimension $n-4$ with associated combinatorial type $\G_{\sigma} = (G_{\sigma}, m_{\sigma})$ such that $G_{\sigma}$ has a unique vertex $v$ satisfying
            \[\val(v) + |m_{\sigma}^{-1}(v)| = 4 \text{ and } N \in m_{\sigma}^{-1}(v).\] 
            \item if $N$ is light, then $\psi_{N, w}$ is the balanced subfan of $\Mwtrop$ that is the union of closed cones $\sigma$ of dimension $n-4$ with associated combinatorial type $\G_{\sigma} = (G_{\sigma}, m_{\sigma})$ such that 
            \begin{enumerate}
                \item $G_{\sigma}$ has a unique vertex $v$ such that 
                \[\val(v) + |m_{\sigma}^{-1}(v)| = 4 \text{ and } N \in m_{\sigma}^{-1}(v).\] 
                \item $\sigma$ is not contained in maximal cones contracted by $\pr_{w}^{\trop}: \Mntrop \to \Mwtrop$. 
            \end{enumerate}
        \end{enumerate}
    \end{repthm}

\begin{example}
\label{exmp:losev-manin}
    When $w=(1^{(2)},\varepsilon^{(n-2)})$, $\Mwtrop$ is the tropical analogue of the Losev--Manin space, studied in \cite{losev2000new}. The theorem implies that $\psi_{N, w} = \varnothing$ for all $N \notin [2]$. Their algebro-geometric counterparts $\psi_{N}$ are also $0$ for all $N \notin [2]$, as a result of \cite[Lemma 5.5]{alexeev_guy_2008}. 
\end{example}

\begin{example}
    Let $w = (1^{(3)}, \ep^{(2)})$ and we compute $\psi_{1, w}$ and $\psi_{4, w}$. 
    The $1$-dimensional cones corresponding to those tropical curves with combinatorial types $\G_{\sigma} = (G_{\sigma}, m_{\sigma})$ possessing a unique vertex $v$ satisfying $\mathrm{val}(v)+|m_{\sigma}^{-1}(v)|=4$ and $4 \in m_{\sigma}^{-1}(v)$ are the $1$-dimensional cones with generators
    \[v_{12}, v_{13}, v_{15}, v_{23}, v_{25}, v_{35}.\] 
    The maximal cones in $M^{\trop}_{0, 5}$ that are contracted under the reduction map $\pr_{w}^{\trop}$ for $w = (1^{(3)}, \varepsilon^{(2)})$ are those for which $v_{45}$ is a primitive generator, which does not correspond to a $w$-stable curve.  These cones contain the codimension-$1$ cones generated by $v_{12}, v_{13}, v_{23}$, which are not present in $\psi_{N, w}$ for $N = 4, 5$. See \cref{fig:contracted-maximals-111ee}. 
    Therefore,  
    \[
    \psi_{4, w} = \text{$1$-dimensional cones with generators } \{v_{15}, v_{25}, v_{35}\}.
    \]
    In contrast for $N = 1$, we obtain:
    \[
    \psi_{1, w} = \text{$1$-dimensional cones with generators } \{v_{23}, v_{24}, v_{25},v_{34}, v_{35}\}.
    \]
    \begin{figure}[h!]
        \centering
        \includegraphics[scale=0.8]{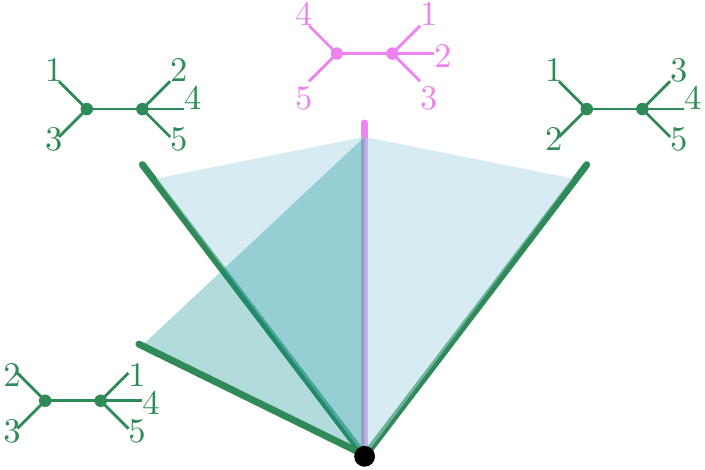}
        \caption{The maximal cones in $M^{\trop}_{0, 5}$ contracted by $\pr_{w}^{\trop}$ for $w = (1^{(3)}, \ep^{(2)})$. }
        \label{fig:contracted-maximals-111ee}
    \end{figure}
\end{example}
    
\begin{proof}[Proof of \cref{thm:psi-class-characterisation}]
For $N \in [n-m]$, the set ${\{S \in \calC_w: N \in S\}}$ is empty. By Definition \ref{def-psiw}, ${\psi_{N, w} = (\pr_{w}^{\trop})_{\ast}(\psi_{N})}$. 

Now suppose $N \in [n] \smallsetminus [n-m]$. Recall that by \cref{def-psiw}
\begin{equation}
\label{eq:psi-combo-proof}
\psiiwtrop =\left(\mathrm{pr}_w^{\trop}\right)_\ast(\psi_N^{\trop}-\sum\limits_{N \in S \in \calC_{w}^2} D_{S}^{\trop}),
\end{equation}
Firstly, we consider the cones in $\sum_{N \in S \in \calC_{w}^2} D_{S}^{\trop}$ that remain under pushforward of $\pr_{w}^{\trop}$. 
For each $S \in \calC_w^2$ containing $N$, let 
\begin{align}
\calP_S &:= \{\sigma: \dim(\sigma) = \dim(X)  - 1, N\in S =P_j(\sigma) \cup P_k(\sigma) \in \calC_w^2, \{j, k\} \subset [4]\}, \\
\calN_S &:= \{\sigma: \dim(\sigma) = \dim(X) - 1,N \in S=P_j(\sigma) \in \calC_w^2\, \text{or}\, N\notin S^c =P_j(\sigma) \in \calC_w^2,j \in [4]\}, 
\end{align} standing for those cones appearing with positive coefficients and with negative coefficients respectively in the tropical boundary divisor $D_S$. 
By \cref{lem:boundary-divisor-weight}, 
\begin{equation}
\label{eq:subtracted-ds-in-codim-one}
    D_S =\sum\limits_{\sigma \in \calP_S}\sigma - \sum\limits_{\sigma\in \calN_S} \sigma.
\end{equation}
We compute $D_S$ under the pushfoward $(\mathrm{pr}^{\trop}_w)_{\ast}$. Fix a codimension-1 cone $\sigma$ appearing with nonzero coefficients in the expression above. Let $\G_{\sigma} = (G_{\sigma}, m_{\sigma})$ be its combinatorial type satisfying that $G_{\sigma}$ has a unique vertex $v$ with $\mathrm{val}(v)+|m^{-1}(v)|=4$ and let the partition \[P_{1}(\sigma)\sqcup P_{2}(\sigma) \sqcup P_{3}(\sigma) \sqcup P_{4}(\sigma) \vdash [n]\] be given in the same manner as in \cref{subsec:tropical-intersection}; we write $P_j(\sigma)$ as $P_{j}$ for simplicity. 

There are two cases, depending on whether the mark $N \in m^{-1}(v)$ or $N \notin m^{-1}(v)$. 

\un{Case 1}: If the marking $N \notin m^{-1}(v)$, then $\{N\} \subsetneq P_i$ for some $i \in [4]$. Since $G$ is a finite tree, there exists a vertex $u \in V(G)$ such that $\val(u) = 1$ and that $m^{-1}(u) \subseteq P_i$. Since $P_i \subseteq S \in \calC_w^2$ by the fact that $\sigma$ has nonzero coefficient in Equation \ref{eq:subtracted-ds-in-codim-one}, we have that
    \[
    \sum_{k \in m^{-1}(u)} w_k \le \sum_{k \in P_i} w_k \le \sum_{k \in S} w_k < 1.
    \]
Thus, \[\val(u) + \sum_{k \in m^{-1}(u)} w_k < 1 + 1 = 2,\] implying that $\G_{\s}$ is not a $w$-stable combinatorial type. Therefore, the pushforward under $\pr_{w}^{\trop}$ of $\sigma$ is empty. 
    
\un{Case 2}: If the marking $N \in m^{-1}(v)$, then without loss of generality let $P_4 = \{N\}$. Since $|S| \ge 2$ and by the fact that $\sigma$ has nonzero coefficient in Equation \ref{eq:subtracted-ds-in-codim-one}, there are then $4$ cases, grouped by whether $S = \{N\} \cup P_i$ for some $i \in [4]$ or $S = \{N\} \sqcup P_i \sqcup P_j$ for $\{i,j\}\subset [4]$, and further subdivided by cardinality considerations. 

    \un{Case 2(a)}: If $S = \{N\} \cup P_i$ for some $i \in [3]$ and $|P_i| \ge 2$, then by the same argument as in Case 1 when $N \notin m^{-1}(v)$, there exists a vertex $u \in V(G)$ such that $m^{-1}(u) \subseteq P_i \subset S \in \calC_{w}^2$ and $\val(u) = 1$. The vertex $u$ is then a witness that $\G_{\sigma}$ is not $w$-stable and thus its image pushforward under $\pr_{w}^{\trop}$ is trivial. 
    
    \un{Case 2(b)}: If $S = \{N\} \cup P_i$ for some $i \in [3]$, $|P_i| = 1$ and one of the remaining parts, say $P_{j}$ for $j \in [4] \smallsetminus \{4, i\}$ has cardinality $1$. If we have that $P_j\subset [n-m]$, we have 
        \[
        \val(v) + \sum_{k \in m^{-1}(v)} w_k = 2 + 2\varepsilon >2. 
        \]
        Then $\G_{\sigma}$ is $w$-stable and thus the pushforward preserves $\sigma$.
        
    \un{Case 2(c)}: If $S = \{N\} \cup P_i$ for some $i \in [3]$, $|P_i| = 1$ and one of the remaining parts, say $P_{j}$ for $j \in [4] \smallsetminus \{4, i\}$ has cardinality $1$. If we have that $P_j\subset[n]\smallsetminus[n-m]$, we have 
        \[
        \val(v) + \sum_{k \in m^{-1}(v)} w_k = 1 + 3\varepsilon < 2. 
        \]
        Then $\G_{\sigma}$ is not $w$-stable and thus the pushforward under $\pr_{w}^{\trop}$ is trivial.
                
    \un{Case 2(d)}: If $S = \{N\} \cup P_i$ for some $i \in [3]$, $|P_i| = 1$, and both of the remaining parts have cardinality greater than $1$. Then the $w$-stability depends on the remaining parts $P_j$ for $j \notin \{i, 4\}$.
    
    \un{Case 2(e)}: If  $S = \{N\} \sqcup P_i \sqcup P_j$ for $\{i,j\}\subset[3]$, it is an easy exercise using similar arguments to see that $\G_{\sigma}$ is not $w$-stable and pushforward under $\mathrm{pr}_w^{\trop}$ is trivial. 

Now, denoting 
\[
\calS_{N, S} := \{\sigma \colon N \in m^{-1}(v), \textrm{if } |P_j(\sigma)| = 1\,\textrm{with }j\in[3],\,\textrm{then } P_j\subset\{[n-m]\}\},
\]
Equation \ref{eq:psi-combo-proof} becomes 
\begin{align}
\label{eq:psi-divisor-of-2}
    \psi_{N,w} &=(\mathrm{pr}_w^{\trop})_\ast \Bigg(\psi_N -\sum_{\sigma\in \calS_{N,S}}  \sigma \Bigg) = \sum_{\sigma \in \calS_{N, S}} \sigma. 
\end{align}
Note that the cones in $\calS_{N,S}$ are exactly those cones with combinatorial types ${\G_{\sigma} = (G_{\sigma}, m_{\sigma})}$ such that the unique vertex $v$ in $G_{\sigma}$ does not carry two light marks. Therefore, they are not contained in a maximal dimensional cone in $\Mntrop$ that is contracted. This proves the result. 
\end{proof}

\subsection{$\psi$-classes on $\Mwtrop$ as Weil divisors of rational functions}
\label{subsec-rationalfunc}
    
In this section, we define a rational function $f_{N,w}$ for each $N \in [n]$ such that the tropical Weil divisor $\div(f_{N,w})$ is a multiple of $\psi_{N, w}$ in $\Mwtrop$. 
\begin{definition}
Let $w$ be heavy/light as in \cref{conv:weight-vector}. Let $I\subset[n]$ of cardinality $1<|I|<n-1$. We define a vector $v_I\in\mathbb{R}^{\binom{n}{2}-\binom{m}{2}}$ as follows. Each coordinate of $\mathbb{R}^{\binom{n}{2}-\binom{m}{2}}$ is indexed by a tuple $T \in \binom{[n]}{2}\smallsetminus\binom{[m]}{2}$. For each coordinate indexed by ${T = \{t_1, t_2\} \not\subseteq  [n] \smallsetminus [n-m]}$, we define
\begin{equation}
    (v_I)_T=\begin{cases}
    1\quad \textrm{if}\,|I\cap T|=1\\
    0\quad \textrm{otherwise}
    \end{cases} \in \mathbb{R}^{\binom{n}{2} - \binom{m}{2}}.
\end{equation}
\end{definition}
For example, when $n = 4$, $w=(1^{(4)})$, $I = 24 \subset [4]$, the vector $v_{24} = (1, 0, 1, 1, 0, 1)$. The motivation for defining such $v_I$ is as follows. The vector $v_I$ is the primitive vector in $M_{0,w}^\trop$ corresponding to a $1$-dimensional cone with two vertices and one bounded edge of length $1$, such that the markings in $I$ on are supported on one endpoint of the bounded edge and $[n] \smallsetminus I$ is supported on the other endpoint. Each $1$-dimensional cone in $M^{\trop}_{0, w}$ is the primitive vector $v_I$ for some $I \subset [n]$. Note that $v_I=v_{[n]\smallsetminus I}$. 

For $N \in [n]$, we define 
\begin{equation}
V_{N,w}=\{v_I\in M_{0,w}^\trop\colon N\notin I \text{ and } |I|=2\} 
\end{equation}
By a similar argument as in \cite[Lemma 2.3]{Kerber_Markwig_2009}, we obtain that for any heavy/light vector $w=(1^{(n-m)},\varepsilon^{(m)})$ and $N\in[n]$ that the span of elements in $V_{N, w}$ is precisely the quotient space $\mathbb{R}_w$, i.e. 
\[
\langle V_{N,w}\rangle=\mathbb{R}_w = \faktor{\mathbb{R}^{\binom{n}{2}-\binom{m}{2}}}{\mathrm{Im}(\phi_w)}.
\]

The next definition/lemma follows from the same ideas as \cite[Definition/Lemma 2.5]{Kerber_Markwig_2009}.

\begin{definitionlemma}
\label{def:positive-rep}
    For any $N \in [n]$, 
    any primitive generator $v_I$ has a unique \textbf{positive representation in $V_{N, w}$}
    \[
    v_I = \sum_{v_{S} \in V_{N, w}} c_{S} v_{S} 
    \]
    satisfying that
    \begin{enumerate}
        \item if $N \in I$, then $S \subseteq [n] \smallsetminus I$; otherwise, $S \subseteq I$; 
        \item for all $v_S \in V_{N, k}$ , we have $c_S\ge0$;
        \item there exists $v_S \in V_{N, k}$ with $c_S=0$.  
    \end{enumerate}
\end{definitionlemma}

For example, for the weight vector $w = (1^{(n)}, \varepsilon^{(m)})$, and $N = 4$, $V_{4, w}$ is $\{v_{12}, v_{13}, v_{23}, v_{25}, v_{35}\}$. The positive representation of $v_{34}$ is $v_{34} = v_{12} + v_{15} + v_{25}$, and the positive representation of $v_{13}$ is $v_{13} = v_{24} + v_{25}$. 

Motivated by this definition, we define the following function that is linear on each cone of $\Mwtrop$. 
\begin{definitionlemma}
For each $N \in[n]$, we define a rational function $f_{N,w}$ on $M^{\trop}_{0, w}$ by 
\[
f_{N, w}(v_I) = \begin{cases}
1 & \text{if } v_I \in V_{N, w}, \\
0 & \text{otherwise}. 
\end{cases}
\] and linearly extend $f_{N,w}$ to $\mathbb{R}^{\binom{n}{2}-\binom{m}{2}}$. In particular, the function $f_{N,w}$ is linear on each cone of $M_{0,w}^\trop$. 
\end{definitionlemma}

For $w=(1^{(n)})$, the linearity of $f_{N, w}$ on each cone was derived in \cite[Lemma 3.3]{Kerber_Markwig_2009} where they defined an analogous rational function $f_N$ on $\Mntrop$. Kerber-Markwig also showed that the Weil divisor $\div(f_{N})$ is a multiple of the tropical $\psi$-class $\psi_{N}$ on $\Mntrop$, which reduces the question of intersecting tropical $\psi$-classes to intersecting Weil divisors with tropical $\psi$-classes. 

\begin{proposition}{{\cite[Proposition 3.5]{Kerber_Markwig_2009}}}
Let $w=(1^{(n)})$. With notation as above, we obtain that
\begin{equation}
    \mathrm{div}(f_{N})=\binom{n-1}{2}\psi_{N}.
\end{equation}
\end{proposition}

Our next theorem is an analogous result on $\Mwtrop$, generalising the above proposition. 
\begin{repthm}{thm:divisor-of-rat-multiple-of-psi}
For $n \ge 4$, $n - m \ge 2$ and $0< \varepsilon < 1/m$, $w=(1^{(n-m)},\varepsilon^{(m)})$, and $N \in [n]$, we have the following equality of tropical divisors
    \[\div(f_{N,w}) = K(N, w) \psi_{N, w},\] 
    where the coefficients $K(N, w)$ depend on $n, m$ as follows. 
    \begin{enumerate}
        \item if $n - m =2$, i.e. the weight vector $w$ contains exactly $2$ heavy weights, then 
        \[
        K(N, w) = \begin{cases} 
        m& \text{if } N \in [n-m],\\
        2m - 2  & \text{otherwise}; 
        \end{cases}
        \]
        \item if $n - m > 2$, then 
        \[
        K(N, w) = \begin{cases} 
        \binom{n - 1}{2} - \binom{m}{2} & \text{if } N \in [n-m],\\
        \binom{n-1}{2} - \binom{m-1}{2}  & \text{otherwise}. 
        \end{cases}
        \]
    \end{enumerate}
\end{repthm}

\begin{remark}
    Note that when $n - m =2$ and $N \in [n] \smallsetminus [n-m]$, the tropical $\psi$-class $\psi_{N,w}$ is empty; see \cref{exmp:losev-manin}. In particular, we have that $\div(f_{N,w})$ is empty.  
\end{remark}

Now we are positioned to adapt the result in \cite[Remark 3.4]{Kerber_Markwig_2009} to the weighted case, allowing us to write the tropical weighted $\psi$-classes as a rational multiple of the tropical Weil divisors. 
\begin{proof}[Proof of Theorem {~\ref{thm:divisor-of-rat-multiple-of-psi}}]
    Fix $N \in [n]$, and denote $M^{\trop}_{0, w}$. For each codimension-$1$ cone $\tau$ in $X$, we compute the weight of $f_{N,w}$ on $\tau$. 
    Recall that 
    \[
    \omega_{f_{N,w}}(\tau) = \sum\limits_{\substack{\sigma \supsetneq \tau \\ \dim \sigma = \dim \tau + 1}} f_{N,w}(\omega(\sigma) u_{\sigma / \tau}) - f_{N,w} \bigg(\sum\limits_{\substack{\sigma \supsetneq \tau \\ \dim \sigma = \dim \tau + 1}} \omega(\sigma) u_{\sigma / \tau}\bigg). 
    \]
    Since $\tau$ has codimension-$1$ in $\Mwtrop$, it parametrises tropical curves with combinatorial types $\G_{\sigma} = (G_{\sigma}, m_{\sigma})$ such that $G_{\sigma}$ has a unique vertex $v$ satisfying $\val(v) + m_{\sigma}^{-1}(v) = 4$. We again obtain a partition 
    \[
    P_1(\sigma) \sqcup P_{2}(\sigma) \sqcup P_3(\sigma) \sqcup P_4(\sigma) \vdash [n]
    \] in the same manner as in \cref{subsec:tropical-intersection}; write $P_i = P_i(\sigma)$ for all $i$. 
    There are at most $3$ top-dimensional cones $\sigma$ containing $\tau$, corresponding to tropical curves with combinatorial types shown in Figure \ref{fig:top-dim-combo-types}. Any such top-dimensional cone $\sigma$ contains $\tau$ in $\Mwtrop$ if and only if the primitive generator $u_{\sigma /\tau}$ correspond to a $w$-stable tropical curve. Equivalently, the top-dimensional cone $\sigma$ contains $\tau$, if and only if the primitive generator $u_{\sigma/\tau} = v_{S_{\sigma}} = v_{[n] \smallsetminus S_{\sigma}}$ exists in $\Mwtrop$, if and only if 
    \[
    \sum\limits_{i \in S_{\sigma}} w_i \ge 1, \text{ and } \sum\limits_{i \in [n] \smallsetminus S_{\sigma}} w_i \ge 1,
    \]
    where $S_{\sigma} \in \{P_1 \cup P_2, P_1 \cup P_3, P_1 \cup P_4\}$. Furthermore, for each $u_{\sigma/\tau}$ present in $\Mwtrop$, the unique positive representation is given by \cref{def:positive-rep}. 
    Therefore, we have that 
    \begin{align}
        \omega_{f_{N,w}}(\tau) &= \sum\limits_{\substack{\sigma \supsetneq \tau \\ \dim \sigma = \dim \tau + 1}} f_{N,w}(\omega(\sigma) u_{\sigma / \tau}) - f_{N,w} \bigg(\sum\limits_{\substack{\sigma \supsetneq \tau \\ \dim \sigma = \dim \tau + 1}} \omega(\sigma) u_{\sigma / \tau}\bigg) \\
        &= \sum\limits_{\substack{\sigma \supsetneq \tau \\ \dim \sigma = \dim \tau + 1}} \omega(\sigma) f_{N,w}\bigg(\sum_{v_S \in V_{N, w}} c_{\sigma, S} v_S\bigg) - f_{N,w} \bigg(\sum\limits_{\substack{\sigma \supsetneq \tau \\ \dim \sigma = \dim \tau + 1}} \omega(\sigma) \bigg(\sum_{v_S \in V_{N, w}} c_{\sigma, S} v_S\bigg)\bigg). 
    \end{align}
    Switching the order of summation and evaluating $f_{N,w}(v_S) = 1$, we have that 
    \begin{align}
        \omega_{f_{N,w}}(\tau) = \sum_{v_S \in V_{N, w}} \sum\limits_{\substack{\sigma \supsetneq \tau \\ \dim \sigma = \dim \tau + 1}}
        \omega(\sigma) c_{\sigma, S} - f_{N,w} \left( \sum_{v_S \in V_{N, w}} \sum\limits_{\substack{\sigma \supsetneq \tau \\ \dim \sigma = \dim \tau + 1}} \omega(\sigma) \left(c_{\sigma, S} v_S\right)\right).
    \end{align}
    We note that $\sum_{v_S\in V_{N,w}} v_S=0$ by a similar argument as in \cite[Lemma 2.4]{Kerber_Markwig_2009}. Therefore, we may obtain the unique positive representation of the argument of $f_{N,w}$ in the second term above, by subtracting $M_{\tau}\sum_{v_S\in V_{N,w}} v_S$, where
    \[
    M_{\tau} \coloneqq \min\limits_{v_S \in V_{N, w}} \bigg(\sum\limits_{\substack{\sigma \supsetneq \tau \\ \dim \sigma = \dim \tau + 1}} \omega(\sigma) c_{\sigma, S}\bigg). 
    \] 
    Therefore, the weight of $f_{N,w}$ on $\tau$ becomes
    \begin{align}
        \omega_{f_{N,w}}(\tau) =  \sum_{v_S \in V_{N, w}} M_{\tau} = K(N, w) M_{\tau}
    \end{align} by linearity of $f_{N,w}$.
    Then Lemma \ref{lem:min-is-binary} and \ref{lem:primitive-vectors-number} gives the desired result. 
    
    \begin{figure}[h!]
        \centering
        \includegraphics{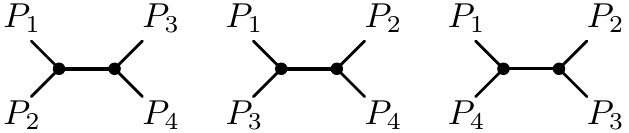}
        \caption{The combinatorial types of the $3$ top-dimensional cones sharing a common codimension-$1$ cone with combinatorial type in Figure \ref{fig:combo-type-codim-1}.}
        \label{fig:top-dim-combo-types}
    \end{figure}
\end{proof}

The following corollary holds immediately from the proof of \cref{thm:divisor-of-rat-multiple-of-psi}.

\begin{corollary}
\label{cor:weight-tau}
    In the situation of \cref{thm:divisor-of-rat-multiple-of-psi}, let $Z$ be an arbitrary $d$-cycle in $\Mwtrop$ for $d \le n-3$ and let $\tau$ be a codimension-1 cone in $Z$. The weight of $\tau$ in the intersection product of $Z$ and $\psi_{N,w}$ is
    \begin{equation}
    \label{eq:weight-tau}
        \omega(\tau)=\mathrm{min}_{v_S\in V_{N,w}}\bigg(\sum_{\substack{\sigma\supsetneq\tau\\\mathrm{dim}(\sigma) =\mathrm{dim}(\tau)+1\\A_{\sigma/\tau}\supset S}} \omega(\sigma)\bigg),
    \end{equation}
    where $A_{\sigma/\tau}\subset [n]\smallsetminus\{N\}$ is the unique set, such that the primitive generator $u_{\sigma/\tau}=v_{A_{\sigma/\tau}}$. 
\end{corollary}

We now prove the two lemmata used in the proof of \cref{thm:divisor-of-rat-multiple-of-psi}.
\begin{lemma}
\label{lem:min-is-binary}
    In the situation of \cref{thm:divisor-of-rat-multiple-of-psi}, for each codimension-$1$ cone $\tau$ in $X$, and $N \in [n]$, the quantity
    \[
    M_{\tau} := \min\limits_{v_S \in V_{N, w}} \bigg( \sum\limits_{\substack{\sigma \supsetneq \tau \\ \dim \sigma = \dim \tau + 1}} \omega(\sigma) c_{\sigma, S} \bigg)= \begin{cases}
    1 & \text{if } \tau \in \psi_{N, w}; \\
    0 & \text{otherwise}.
    \end{cases}
    \]
\end{lemma}

\begin{proof}
    Since $\tau$ has codimension $1$ in $X$, it parametrises combinatorial types with a unique vertex $v$ such that $\val(v) + m^{-1}(v) = 4$. We obtain a partition $A \sqcup B \sqcup C \sqcup D \vdash [n]$ in the manner of \cref{subsec:tropical-intersection} and without loss of generality assume that $N \in A$. There are two cases:
    
    \un{Case 1}: $\{N\} = A$. For any $v_S \in V_{N, w}$, we have that $S$ is contained in one of ${\{B \cup C, B \cup D, C \cup D\}}$. Furthermore, by $w$-stability and without loss of generality, there exists $b \in B$ and $c \in C$ such that $v_{\{b, c\}} \in V_{N, w}$. Consider all $\sigma \supsetneq \tau$ with $\dim \sigma = \dim \tau + 1$; the number of times $v_{\{b, c\}}$ appears in $S_{\sigma/\tau}$ is $1$ where $S_{\sigma/\tau} \subseteq [n] \smallsetminus \{N\}$ such that the primitive generator $u_{\sigma/\tau} = v_{S_{\sigma/\tau}}$. Thus $\{b, c\}$ is a witness of the minimum value of 
    \[
    \sum\limits_{\substack{\sigma \supsetneq \tau \\ \dim \sigma = \dim \tau + 1}} \omega(\sigma) c_{\sigma, S} = 1. 
    \]
    
    \un{Case 2}: $\{N\} \subsetneq A$. Then there exists at least one other element, denoted by $j$ in $A$. For any $v_S \in V_{i, w}$ such that $j \in S$, we have that $S \cap A \ne \varnothing$, and thus $S$ is contained in one of $\{B \cup C, B \cup D, C \cup D\}$. The set $v_S$ is a witness of the minimum value of  
    \[
    \sum\limits_{\substack{\sigma \supsetneq \tau \\ \dim \sigma = \dim \tau + 1}} \omega(\sigma) c_{\sigma, S} = 0,
    \]
    giving $M_{\tau} = 0$ in this case. 
    In both cases, we used the fact that for any $\sigma$ satisfying $\sigma \supsetneq \tau$, $\dim \sigma = \dim \tau + 1$, $c_{\sigma, S} = 1$ if ${S \subseteq S_{\sigma/\tau}}$ and $0$ otherwise. Here, $S_{\sigma/\tau}$ is the unique subset of $[n]\smallsetminus \{N\}$ such that ${u_{\sigma/\tau} = v_{S_{\sigma/\tau}}}$. 
\end{proof}

\begin{lemma}
\label{lem:primitive-vectors-number}
    For $n \ge 4$, $n - m \ge 2$ and $0< \varepsilon < 1/m$, $w=(1^{(n-m)},\varepsilon^{(m)})$, we have that  
    \begin{enumerate}
        \item if $n - m =2$, i.e. the weight vector $w$ contains exactly $2$ heavy weights, then 
        \[
        K(N, w) = \begin{cases} 
        m & N \in [2],\\
        2m - 2 & \text{otherwise}; 
        \end{cases}
        \]
        \item if $n - m > 2$, then 
        \[
        K(N, w) = \begin{cases} 
        \binom{n - 1}{2} - \binom{m}{2} & \text{if } N \in [n-m],\\
        \binom{n-1}{2} - \binom{m-1}{2}  & \text{otherwise}. 
        \end{cases}
        \]
    \end{enumerate}
\end{lemma}

\begin{proof}
    Fix $N \in [n]$. 
    Recall that $K(N, w)$ is the number of subsets $S \subseteq [n]$ such that $|S| = 2$, $N \notin S$, and the tropical curve with only one bounded edge supporting $S$ on one endpoint is $w$-stable, i.e.
    \[
    \sum\limits_{N \in S} w_i \ge 1, \text{ and } \sum\limits_{N \in [n] \smallsetminus S} w_i \ge 1. 
    \] 
    For (i), suppose $N \in [2]$. Any $S \in V_{N, w}$ is precisely of the form $([2] \smallsetminus N) \cup j$ for a light weight index $j \in [n]\smallsetminus [2]$. Thus $K(N, w)$ is the number of ways of choosing a light weight index, and is $m$. Now suppose $N \in [n]\smallsetminus [2]$;  then any $S \in V_{N, w}$ is of the form $j \cup k$ for $j \in [2]$ and $k \in [n] \smallsetminus (N \cup [2])$, giving the count $K(N, w) = 2(m-1) = 2m - 2$. 
    
    For (ii), any $2$-subsets of $[n]$ not containing $N$ can be $V_{N, w}$ except for those contained in ${[n] \smallsetminus [n-m]}$. In the case when $N \in [n-m]$, this excludes $\binom{m}{2}$ subsets of $[n]\smallsetminus [n-m]$; when $N \in [n]\smallsetminus [n-m]$, this excludes the $\binom{m}{2}$ contained in $[n]\smallsetminus [n-m]$ but not containing $N$. 
    \end{proof}

\section{Tropical local multiplicities and intersection numbers}
\label{sec:GW-invariants}
In this section, we prove \cref{thm:main-later}, and we first prepare by introducing the notion of tropical local multiplicity at each vertex of a tropical curve. 
\begin{definition}
    Let $w$ be as in \cref{conv:weight-vector} and a set $S \subseteq [n]$, a partition $P = P_1 \sqcup \cdots \sqcup P_r$ of $S$ is \textbf{totally $w$-unstable} if 
    \[
    \sum_{i \in P_j} w_i \le 1, 
    \] for all $j \in [r]$. 
\end{definition} 
We denote the set of all totally $w$-unstable partition of $S$ by $\calP_w(S)$. 

We will use the following definitions involving set partitions. 
\begin{definition}
\label{def:set-admissible-partitions}
    Given a partition 
    \[
    P = P_1 \sqcup \cdots \sqcup P_r \vdash M
    \] of a set $M$ and a subset $S \subseteq M$, 
    the partition $P$ is called \textbf{$S$-admissible} if there exists a subset $I \subseteq [r]$ such that $\sqcup_{i \in I} P_i = S$.  
\end{definition}
As an example, consider the partition $P = \{1, 23, 4\} \vdash [4]$ and the subset $S_1 = 12$ and $S_2 = 234$. Then $P$ is $S_2$-admissible but not $S_1$-admissible. 

\begin{definition}
\label{def:P-sequence-of-K}
    Given a sequence of numbers $K = \{k_i\}_{i \in \mathbb{Z}_{> 0}}$ and a partition 
    $P$ of a subset of $\mathbb{Z}_{> 0}$, 
    the \textbf{$P$-sequence of $K$} is 
    \[
    K(P) := \bigg\{1 - |P_i| + \sum_{j \in P_i} k_j\bigg\}_{i \in \mathbb{Z}_{> 0}}.
    \]
\end{definition}

Next, we define the tropical local multiplicities at the vertices of abstract tropical curves.

\begin{definition}
\label{def:gw-mult}
 \normalfont Let $\sigma$ be a cone in $\Mwtrop$ with combinatorial type $\G_{\sigma} = (G_{\sigma}, m_{\sigma})$ and let $v \in V(G_{\sigma})$. We define the \textbf{tropical local multiplicity at the vertex $v$ of $\sigma$} as
        \[
        \TLM_{\sigma}(v) := \sum_{P \in \calP_{w}(m_{\sigma}^{-1}(v))} (-1)^{|m_{\sigma}^{-1}(v)| - \ell(P)} \binom{\sum_{i} K(P)_i}{K(P)_1\cdots K(P)_{\ell(P)}}. 
        \]
\end{definition}
\begin{remark}
\label{rmk:gw-combo-only}
    Note that the tropical local multiplicity at a vertex or the product of all the tropical local multiplicities at all vertices of a given tropical curve completely depends on the combinatorial type of the tropical curve. Importantly, \cref{lem:nonnegative} implies that it is always nonnegative. 
\end{remark}

We are now ready to formulate \cref{thm:main-later} precisely.
\begin{repthm}{thm:main-later}
    Let $w \in (\Q \cap (0, 1])^n$ be heavy/light, and $K = (k_1, \ldots, k_n) \in (\mathbb{Z}_{\ge 0})^n$. The intersection product $\prod\limits_{i = 1}^{n} \psi_{i, w}^{k_i}$ is the weighted subfan of $\Mwtrop$ consisting of closures of the cones of codimension $\sum_i k_i$ satisfying the following conditions:
    \begin{enumerate}
        \item For each maximal cone $\sigma$ in $\prod\limits_{i = 1}^{n} \psi_{i, w}^{k_i}$ with combinatorial type $\G_{\sigma} = (G_{\sigma}, m_{\sigma})$, and for each vertex $v \in V(G_{\sigma})$, we have
        \[
        \val(v) + |m_{\sigma}^{-1}(v)| = 3 + \sum_{i \in m_{\sigma}^{-1}(v)} k_i. 
        \]
        \item The weight of a maximal cone $\sigma$ in $\prod\limits_{i = 1}^{n} \psi_{i, w}^{k_i}$ with combinatorial type $\G_{\sigma} =(G_{\sigma}, m_{\sigma})$ is the product of the tropical local multiplicities at all vertices of $G_{\sigma}$, i.e.
        \begin{align}
        \label{eq:tau-weight}
        \omega(\sigma) &= \prod_{v \in V(G_{\sigma})} \TLM_{\sigma}(v)\\
        &= \prod_{v \in V(G_{\sigma})} \sum_{P \in \calP_{w}(v)} (-1)^{|m_{\sigma}^{-1}(v)| - \ell(P)} \binom{\sum_{i} K(P)_i}{K(P)_1\cdots K(P)_{\ell(P)}}.
    \end{align}
    \end{enumerate}
\end{repthm}

\begin{proof}[Proof of Theorem \ref{thm:main-later}]
    We prove by induction on the number of intersecting weighted tropical $\psi$-classes, or equivalently, $\sum_i k_i$.
    To proceed, we set the following notations for convenience.
    \begin{enumerate}
        \item If $\calP(M)$ is a set of partitions of a set $M$, we denote the set of $S$-admissible partitions by $\calP^{S}(M)$.  
        \item Given any partition $P = P_1 \sqcup \cdots \sqcup P_r$ of $[n]$ and $I \subseteq [r]$, we set $\supp(\{P_i\}_{i \in I}) := \cup_{i \in I} P_i$, i.e. the support of some parts is the union of those parts. 
        \item For any cone $\sigma$, write the combinatorial type of $\sigma$ as $\G_{\sigma} = (G_{\sigma}, m_{\sigma})$. Write $V(G_{\sigma})$ and $E(G_{\sigma})$ as the vertex and the edge set of $G_{\sigma}$. 
        \item In the inductive step, denote $\I =\prod \psi_{i}^{k_i}$ and $\I_N = \psi_{N} \prod \psi_{i}^{k_i}$ for $N \in [n]$.
        \item Given a codimension-$1$ cone $\tau$ in $\I$ and $v_{T} \in V_{N, w}$, define  
        \[
        \calS(\tau, v_T) := \{\sigma: \sigma \supsetneq \tau, \sigma \in \I, \text{ and } T\subset S_{\sigma/\tau}\},
        \]
        where for each $\sigma$, $S_{\sigma/\tau}$ is the unique set $[n] \smallsetminus \{N\}$ such that $u_{\sigma/\tau} = v_{S_{\sigma/\tau}}$ as usual. 
        \item Given a codimension-$1$ cone $\tau$ in $\I$ and $v_T \in V_{N, w}$, let 
        \[
        \Sigma(\tau, v_T) := \sum_{\sigma \in \calS(\tau, V_{T})} \omega(\sigma).
        \]
    \end{enumerate}
    The base case is when $\sum_{i} k_i = 1$; there is $i \in [n]$ such that $k_i = 1$, $k_j = 0$ for all $j \ne i$ and then the combinatorial description of a weighted tropical $\psi$-class in \cref{thm:psi-class-characterisation} implies the desired result. For the inductive step, we assume that \cref{thm:main-later} holds for $\I$. We compute $\I_N$ for $N \in [n]$ and there are two parts to prove. 
    
    For part (i), suppose $\tau$ is a codimension-$1$ cone in $\I$ and thus $\G_{\tau}$ is the edge-contraction of an edge $e \in E(G(\sigma))$ for some maximal dimensional cone $\sigma$ in $\I$. By the induction hypothesis, the vertex $v$ in $G_{\tau}$ as the result of the edge-contraction of $e$ satisfies
    \begin{align}
        \val(v) + |m_{\tau}^{-1}(v)| &= (\val(v_1) + |m_{\sigma}^{-1}(v_1)| - 1) + (\val(v_2) + |m_{\sigma}^{-1}(v_2)| - 1) \\
        &= \bigg(3 - 1 + \sum_{i \in m_{\sigma}^{-1}(v_1)} k_i \bigg) + \bigg(3 - 1 + \sum_{i \in m_{\sigma}^{-1}(v_2)} k_i  \bigg) \\
        &= 4 + \sum_{i \in m_{\sigma}^{-1}(v_1) \cup m_{\sigma}^{-1}(v_2)} k_i \\
        &= 4 + \sum_{i \in m_{\tau}^{-1}(v)} k_i .
    \end{align}
    Furthermore, following a similar argument in the proof of \cite[Theorem 4.1]{Kerber_Markwig_2009}, we have that $N \in m^{-1}(v)$. Then the above can be rewritten as 
    \begin{align}
        \val(v) + |m_{\tau}^{-1}(v)| &= 3 + (1 + k_N) + \sum_{\substack{i \in m_{\tau}^{-1}(v), \\ i \ne N}} k_i 
    \end{align} as desired and proving part (i). 
    
    For part (ii), let us restate \cref{thm:main-later} (ii) as follows. Firstly, note that for each $\sigma$ such that $\tau \subsetneq \sigma \in \I$, writing the contracted edge as $e \in E(G_{\sigma})$ with endpoints $v_1$ and $v_2$, we can assume without loss of generality that $N \in m_{\sigma}^{-1}(v_1)$. Furthermore, by \cref{cor:weight-tau} and recalling set notations, 
        \[
        \omega(\tau) = \min_{V_{T} \in V_{N, w}} \bigg(\sum_{\sigma \in \calS(\tau, v_T)} \omega(\sigma)\bigg) = \min_{V_T \in V_{N, w}} \Sigma(\tau, v_T). 
        \] 
    Let $T^{\ast} = \{t_1, t_2\} \subset ([n] \smallsetminus \{N\})$, such that the above minimum is achieved. 
    
    We make two observations about such $T^{\ast}$. Firstly, because $v_{T^{\ast}} \in V_{N, w}$, we may assume $w_{t_1} = 1$ without loss of generality. 
    Secondly, we may assume that $t_1, t_2$ are marks on distinct connected components of $G_{\sigma}\smallsetminus e$. Otherwise, suppose $t_1,t_2$ are marks on the same connected component of $G_{\sigma}\smallsetminus e$ and take $t_2'\neq N$ to be any mark on another connected component, then
    \begin{equation}
        \calS(\tau,v_{T^\ast})\subseteq \calS(\tau,v_{\{t_1,t_2'\}}).
    \end{equation}
    By \cref{lem:nonnegative}, we have $\omega(\sigma) \ge 0$ for any $\sigma \in \calS(\tau, v_S)$ for $V_S \in V_{N, w}$. This implies that 
    \begin{equation}
        \Sigma(\tau,v_{T^\ast})\le\Sigma(\tau,v_{\{t_1,t_2'\}}).
    \end{equation}
    Therefore, we hereafter take $T^{\ast} = \{t_1, t_2\} \subset ([n] \smallsetminus \{N\})$, such that $t_1,t_2$ are contained in distinct connected components of $G_{\sigma} \smallsetminus e$ and $w_{t_1}=1$, and we compute $\omega(\tau) = \Sigma(\tau, v_{T^{\ast}})$.
    Now we want to show that
    \begin{equation}
    \label{eq:rewriting-the-min}
    \Sigma(\tau, v_{T^{\ast}}) = \prod_{v \in V(G_{\tau})} \TLM_{\tau}(v). 
    \end{equation}
    We further simplify the above by making the following observation: recall that $\Sigma(\tau, v_{T^{\ast}}) = \sum_{\sigma \in \calS(\tau, v_{T^{\ast}})} \omega(\sigma)$. By the inductive hypothesis, the weight of each $\sigma$ equals 
    \[
    \omega(\sigma) = \prod_{u \in V(G_{\sigma})} \TLM_{\sigma}(u). 
    \]
    By \cref{rmk:gw-combo-only}, any $\sigma \in \calS(\tau, v_{T^{\ast}})$ satisfies that 
    \begin{equation}
    \prod_{\substack{u \in V(G_{\sigma}), \\u \neq v_1, v_2}} \TLM_{\sigma}(u) = \prod_{\substack{w \in V(G_{\tau}), \\w \ne v}} \TLM_{\tau}(w).
    \end{equation}
    Cancelling out these terms on both sides of \cref{eq:rewriting-the-min}, proving \cref{eq:rewriting-the-min} amounts to showing that  
    \begin{equation}
    \label{eq:two-GWI-become-one-GWI}
    \sum_{\sigma \in \calS(\tau, V_{T^{\ast}})}  \TLM_{\sigma}(v_1) \TLM_{\sigma}(v_2) = \TLM_{\tau}(v). 
    \end{equation}
    
    To proceed, we study the range of the summation above, namely, the set $\calS(\tau, V_{T^{\ast}})$. The cones in $\calS(\tau, V_{T^{\ast}})$ can be described completely combinatorially: for any such $\sigma$, again writing the contracted edge as $e$ with endpoints $v_1$ and $v_2$, we have that 
    \begin{enumerate}
        \item for the markings $m_{\tau}^{-1}(v)=m_\sigma^{-1}(v_1)\sqcup m_\sigma^{-1}(v_2)$; 
        \item the marking $N \in m_{\sigma}^{-1}(v_1)$ by assumption; and 
        \item $T^{\ast}$ is contained the same component of $G_\sigma\smallsetminus e$ as $v_2$. 
    \end{enumerate}
    Therefore, writing $m_{\sigma}^{-1}(v_1)$ as $S$ and $m_{\tau}^{-1}(v)$ as $M$, the left hand side of \cref{eq:two-GWI-become-one-GWI} is 
    \begin{align}
    \label{eq:GWI-product-over-S}
        \sum_{\sigma \in \calS(\tau, V_{T^{\ast}})} \TLM_{\sigma}(v_1) \TLM_{\sigma}(v_2) = \sum_{\substack{N \in S \subseteq M,\\ T^{\ast} \subset M\smallsetminus S}}\, \sum_{\substack{\sigma \in \I, \\ S = m_{\sigma}^{-1}(v_1)}}
        \TLM_{\sigma}(v_1) \TLM_{\sigma}(v_2).
    \end{align}
    
    We now analyse the right hand side of \cref{eq:GWI-product-over-S}.
    For each $S$, the number of maximal cones $\sigma$ satisfying that $S$ is the set of markings supported on $v_1$ is precisely the number of distinct ways of distributing edges adjacent to $v$. The total number of edges adjacent to $v$ is 
    \[
    \val(v) = 4 - |M| + \sum_{i \in M} k_i . 
    \] 
    Furthermore, for any such cone $\sigma$, by our assumption that $t_1, t_2$ are on different components of $G_{\sigma} \smallsetminus e$, there are precisely $2$ edges on the same components with $t_1$ and $t_2$ respectively that are not adjacent to $v_1$. Thus the total number of edges adjacent to $v$ in $G_{\tau}$ that could be adjacent to $v_1$ in $G_{\sigma}$ is  
    \[
    \val(v) - 2 = 4 - 2 - |M| + \sum_{i \in M} k_i = 2 - |M| + \sum_{i \in M} k_i. 
    \]
    The total number of edges adjacent to $v_1$ and existent in $G_{\tau}$ is equal to the valence of $v_1$ subtracting $1$ (disregarding the edge $e$):
    \[
    \val(v_1) - 1  = 3 - 1 - |S| + \sum_{i \in S} k_i = 2 - |S| + \sum_{i \in S}k_i. 
    \]
    Therefore, the total number of maximal cones $\sigma$ in $\I$ such that $S = m^{-1}_{\sigma}(v_1)$ is precisely the binomial coefficient
    \[
    \binom{2 - |M| + \sum_{i \in M} k_i}{2 - |S| + \sum_{i \in S} k_i}. 
    \]
    Moreover, for each such $S$ and $\sigma$, the product of tropical local multiplicities at $v_1$ and $v_2$ is
    \begin{align}
    \label{eq:product-into-two-partitions}
        \left( \sum_{P_1 \in \calP_{w}(S)} (-1)^{|S| - \ell(P_1)} \binom{\sum_{i} K(P_1)_i}{K(P_1)_1,\ldots, K(P_1)_{\ell(P_1)}}\right) \left( \sum_{P_2\in \calP_{w}(S^{c})} (-1)^{|S^{c}| - \ell(P_1)} \binom{\sum_{i} K(P_2)_i}{K(P_2)_1, \ldots, K(P_2)_{\ell(P_2)}}\right).
    \end{align}
     Next, we observe that 
     \[
    \left\{P_1 \cup P_2: P_1 \in \calP_{w}(S), P_2 \in \calP_{w}(S^c), N \in S,t_1 \notin S, t_2 \notin S, S \subseteq M\right\} 
    \] is precisely the $S$-admissble partitions in $\calP_{w}(M)$, such that $t_2$ and $N$ are not contained in the same part of the partition. 
    
    To summarise, the right hand side of \cref{eq:GWI-product-over-S} becomes 
    \begin{align}
    \label{eq:final-split-up}
        \sum_{P,P_1,P_2} (-1)^{\lambda} &\binom{2 - |M| + \sum_{i \in M} k_i}{2 - |\cup_{j \in \ell(P_1)} P_{1, j}| + \sum_{i \in S} k_i}\binom{\sum_{i} K(P_1)_i}{K(P_1)_1, \ldots,  K(P_1)_{\ell(P_1)}}\binom{\sum_{i} K(P_2)_i}{K(P_2)_1,\ldots, K(P_2)_{\ell(P_2)}},
    \end{align}
    where 
    \begin{enumerate}
    \item $P\in\calP_w(M)$ with $t_2,N$ not in same part of $P$; 
    \item $P_1 \cup P_2 = P$ such that $N\in\supp(P_1)$ and $t_1,t_2\not\in\mathrm{Supp}(P_1)$; 
    \item $\lambda=|M|-\ell(P)$.
    \end{enumerate}
    
    We simplify \cref{eq:final-split-up} as follows. Firstly, we note the following equality of the binomial coefficients
    \[
    \binom{2 -|M| + \sum_{i \in M} k_i}{2 - |\supp(P_1)| + \sum_{i \in S} k_i} = \binom{2 -\ell(P)+ \sum_{i \in M} k_i - (|M| - \ell(P)))}{2 -\ell(P) + \sum_{i \in S} k_i - (|\supp(P_1)| - \ell(P_1))}. 
    \]
    Secondly, we observe that 
    \[
    \sum_{i \in M} k_i - (|M| - \ell(P)) = \sum_{P_i \in P} \bigg(\sum_{j \in P_i} k_j\bigg) - (|P_i| - 1)  = \sum_i K(P)_i, 
    \]
    and that 
    \[
    \sum_{\substack{N \in \supp(P_1),\\ t_1, t_2 \notin \supp{P_1}}} k_i - (|\supp(P_1)| - \ell(P_1)) = \sum_{P_{1,i} \in P_1} \bigg(\sum_{j \in P_{1,i}} k_j\bigg) - (|P_{1,i}| - 1) = \sum_{i}K(P_1)_i, 
    \]
    Third, observe that for a fixed $P \in \calP_{w}(M)$, and any pair of $P_1$ and $P_2$ such that $\{P_1\} \cup \{P_2\} = P$, the terms indexed by $(P_1, P_2)$ in the right hand side of \cref{eq:final-split-up} have the same denominator 
    \[
    \prod_{i} K(P_1)_i! \prod_{j}K(P_2)_j! = \prod_{i} K(P)_i! 
    \]
    Therefore, Equation \ref{eq:final-split-up} equals 
    \begin{align}
        \sum_{\substack{P \in \calP_w(M)\\t_2,N\,\textrm{not in}\\
        \textrm{same part of }P}}\frac{ (-1)^{|M| -\ell(P)}}{\prod_{i} K(P)_i!} \sum_{\substack{\{P_1\} \cup \{P_2\} = P\\N\in\supp(P_1)\\t_1,t_2\not\in\mathrm{Supp}(P_1)}} \binom{2 - \ell(P) + \sum_i K(P)_i}{2 - \ell(P_1) + \sum_{j} K(P_1)_j} \bigg(\sum_i K(P_1)_i \bigg)!\bigg(\sum_j K(P_2)_j\bigg)!.
    \end{align}
    By \cite[Equation (3)]{Kerber_Markwig_2009}, for fixed $P$ the second summation becomes
    \[
    \sum_{\{P_1\} \cup \{P_2\} = P} \binom{\sum_i K(P)_i - \ell(P) + 2}{\sum_{j} K(P_1)_j - \ell(P_1) + 2} \bigg(\sum_i K(P)_i \bigg)! \bigg(\sum_j K(P_1)_j \bigg)! = \frac{(\sum_{i} K(P)_i + 1)!}{(K(P^{\ast})_1 + 1)}
    \] where $P^{\ast}$ is the unique part in $P$ such that $N \in P^{\ast}$. 
    Putting everything together, Equation \ref{eq:final-split-up} becomes 
    \begin{align}
    \label{equ:finalresolution}
        \sum_{\substack{P \in \calP_w(M)\\t_2,N\,\textrm{not in}\\
        \textrm{same part of }P}} (-1)^{|M| - \ell(P)} \frac{(\sum_i K(P)_i + 1)!}{(K(P^{\ast})_1 + 1) \prod_{i}K(P)_i!}.
    \end{align} 
    By \cref{lem:not-colliding}, we have that \cref{equ:finalresolution} achieves minimum for such $T^\ast$, where $t_2$ and $N$ can never be in the same part of $P$. For such a choice of $T^\ast$, we obtain that \cref{equ:finalresolution} becomes
    \begin{align}
        \sum_{P \in \calP_w(M)} (-1)^{|M| - \ell(P)} \frac{(\sum_i K(P)_i + 1)!}{(K(P^{\ast})_1 + 1) \prod_{i}K(P)_i!}=\TLM_{\tau}(v)
    \end{align} 
    as desired.
\end{proof}

One immediate corollary is that, when $w = (1^{(n)})$, we recover the result of Kerber-Markwig on $\Mntrop$. In the case of top dimension, this also confirms Katz' expectation that the tropical intersection product of $\psi$-classes coincides with their classical counterparts, computed in \cite[Theorem 7.9]{alexeev_guy_2008}. See \cref{cor-kerbmark} and \cref{cor-tropalg}. 

\appendix
\section{Technical lemmata for \cref{thm:main-later}}
We now prove two lemmata, giving a characterisation of the primitive generators for maximal cones $\sigma$ that minimise the sum appearing in the inductive step in the main theorem. \cref{lem:nonnegative} is a recursive formula of the multinomial coefficient and a concrete application of the inclusion-exclusion principle. \cref{lem:not-colliding} is a direct application of \cref{lem:nonnegative} to yield the desired characterisation. 
\begin{lemma}
\label{lem:nonnegative}
    Let $w$ be as in \cref{conv:weight-vector} and $K = (k_1, \ldots, k_n) \in (\mathbb{Z}_{\ge 0})^n$. Then 
    \begin{equation}
    \label{eq:nonnegative}
        \sum_{P\in \calP_w(S)}(-1)^{|S|-\ell(P)}\binom{\sum_{i} K(P)_i}{K(P)_1, \ldots, K(P)_{\ell(P)}}\ge 0.
    \end{equation}
\end{lemma}

\begin{proof}
    The proof proceeds by induction on $n + \sum_i k_i$. In the base case when $n + \sum_{i} k_i = 1$, we must have that $n = 1$, $k_i = 0$ for all $i$, and thus the equality holds trivially. Assume that the result holds for all $n' \ge 1, k_1',\dots,k_n' \in \mathbb{Z}_{\ge 0}$, such that $n' + \sum k_i'< n + \sum k_i$.  

    To compute the left hand side of (\ref{eq:nonnegative}), for each $P \in \calP_{w}(S)$, we apply the recursive formula for multinomial coefficients:
    \begin{align}
    \binom{\sum_{i} K(P)_i}{K(P)_1, \ldots, K(P)_{\ell(P)}} = \sum_{j=1}^{\ell(P)}\binom{(\sum_{i}K(P)_i) - 1}{K(P)_1, \ldots, K(P)_j-1, \ldots, K(P)_{\ell(P)}}.
    \end{align}
    Then the left hand side of (\ref{eq:nonnegative}) becomes 
    \begin{align}
        &\sum_{P\in \calP_w(S)}(-1)^{|S|-\ell(P)}\binom{\sum_i K(P)_i}{K(P_1), \ldots, K(P)_{\ell}} \\
        &=\sum_{P\in \calP_w(S)}(-1)^{|S|-\ell(P)}\sum_{j=1}^{\ell(P)}\binom{\sum_{i}K(P)_i - 1}{K(P)_1, \ldots, K(P)_j-1, \ldots, K(P)_{\ell(P)}}. 
    \end{align}
   
    To proceed, we prepare new notations as follows. 
    \begin{enumerate}
        \item For $\varnothing \subsetneq I\subset S\smallsetminus[n-m]$, set $S_I := S\smallsetminus I\cup \{\bullet_I\}$. 
        \item Set $k_{\bullet_I} :=\sum_{i\in I}k_i-|I|$ and $K_I$ be the sequence obtained from $K$ by deleting $k_i$ indexed by $I$ and appending $k_{\bullet_I}$. 
        \item Set $w_I :=(1^{(n-m)},\varepsilon^{(m-|I|+1)})$ to be the weight vector obtained from $w$ by deleting the weights indexed by $I$ and appending the weight of the mark $\bullet_I$ as $\varepsilon$. 
    \end{enumerate}
    
    Firstly, since $|S_I| = n - |I| + 1$ and $|I| \ge 1$, we have \[\sum_{i\in S_I}k_i+|S_I|=\sum_{i\in S}k_i-|I|+|S_I|=\sum_{i\in S}k_i+n-2|I|+1<\sum_{i\in S}k_i+n.\] 
    
    Secondly, we have the following \underline{claim}:
    \begin{align}
    \label{equ:inclexcl}
        &\sum_{P\in \calP_w(S)}(-1)^{|S|-\ell(P)}\sum_{j=1}^{\ell(P)}\binom{\sum_{i}K(P)_i}{K(P)_1, \ldots, K(P)_j-1, \ldots, K(P)_{\ell(P)}} \\
        &=\sum_{\emptyset\subsetneq I\subset S\smallsetminus[n-m]}\sum_{P\in\mathcal{P}_{w_I}(S_I)}(-1)^{|S_I|-\ell(P)}\binom{\sum_{j} K_{I}(P)_j}{K_I(P)_1, \ldots, K_I(P)_{\ell(P)}}.
    \end{align}
    To prove the claim, we fix $P\in \mathcal{P}_{w}(S)$, $j\in [\ell(P)]$, and compare the coefficients of the multinomial coefficient 
    \begin{equation}
    \label{equ:binom}
        \binom{\sum_{i} K(P)_i}{K(P)_1, \ldots, K(P)_j - 1, \ldots, K(P)_{\ell(P)}}
    \end{equation} on both sides. 
    Observe that it appears with coefficient $(-1)^{|S|-\ell(P)}$ on left hand side of \cref{equ:inclexcl}. On the right hand side of \cref{equ:inclexcl}, for each nonempty $I \subset P_j$, the partition
    \[
    P_1 \sqcup \cdots \sqcup \bigg(P_{j}\smallsetminus I\cup\{\bullet_I\}\bigg) \sqcup \cdots \sqcup P_{\ell(P)}
    \] is in $\mathcal{P}_{w_I}(S_I)$. Thus, the multinomial coefficient appears on the right hand side of \cref{equ:inclexcl} with coefficient
    \begin{align}
        \sum_{\emptyset\subsetneq I\subset P_{j}}(-1)^{|S_I|-\ell(P)} &=\sum_{\emptyset\subsetneq I\subset P_j}(-1)^{|S|-|I|+1-\ell(P)}\\
        &=(-1)^{|S|-\ell(P)}\sum_{\emptyset\subsetneq I\subset P_{j}}(-1)^{|I|+1} \\
        &=(-1)^{|S|-\ell(P)}\sum_{i=1}^{|P_{j}|}(-1)^{i+1}\binom{|P_{j}|}{i} \\
        &=(-1)^{|S|-\ell(P)}, 
    \end{align} thus proving the claim. The last equality is an application of the inclusion-exclusion principle in terms of binomial coefficients: for any integer $m > 0$,
    \[
    \sum_{i = 1}^{m} (-1)^{i+1} \binom{m}{i}=1. 
    \]
    
    Lastly, by the induction hypothesis, every summand
    \begin{align}
        \sum_{P\in\mathcal{P}_{w_I}(S_I)}(-1)^{|S_I|-\ell(P)}\binom{\sum_j K_I(P)_j}{K_I(P)_1, \ldots, K_{I}(P)_{\ell(P)}},
    \end{align}
    of the right hand side of \cref{equ:inclexcl} is nonnegative, and thus the entire sum is nonnegative. Therefore, we obtain that
    \begin{equation}
        \sum_{P\in \calP_w(S)}(-1)^{|S|-\ell(P)}\binom{\sum_{i} K(P)_i}{K(P)_1, \ldots, K(P)_{\ell(P)}} \ge 0,
    \end{equation}
    as desired.
\end{proof}

The following lemma is a direct application of \cref{lem:nonnegative}.

\begin{lemma}
\label{lem:not-colliding}
    Let $w$ be as in \cref{conv:weight-vector} and $S = [n]$. Let $N \in [n]$, ${T = \{t_1, t_2\} \subseteq ([n]\smallsetminus \{N\})}$ with $t_1\in [n-m]$, then the following inequality holds:
    \[
    \sum_{\substack{P\in\mathcal{P}_w(S)}}(-1)^{|S|-\ell(P)}\binom{\sum_i K(P)_i}{K(P)_1, \ldots, K(P)_{\ell(P)}}\le\sum_{\substack{P\in\mathcal{P}_w(S)\\t_2,N\,\textrm{not in}\\
    \textrm{same part of}\,P}}(-1)^{|S|-\ell(P)}\binom{\sum_i K(P)_i}{K(P)_1, \ldots, K(P)_{\ell(P)}}.
    \]
\end{lemma}

\begin{proof}
    We first observe that the left hand side is
    \begin{align}
        &\sum_{\substack{P\in\mathcal{P}_w(S)\\t_2,N\,\textrm{not in}\\
        \textrm{same part of}\,P}}(-1)^{|S|-\ell(P)}\binom{\sum_i K(P)_i}{K(P)_1, \ldots, K(P)_{\ell(P)}}\\
        &+\sum_{\substack{P\in\mathcal{P}_w(S)\\t_2,N\,\textrm{in}\\
        \textrm{same part of}\,P}}(-1)^{|S|-\ell(P)}\binom{\sum_i K(P)_i}{K(P)_1, \ldots, K(P)_{\ell(P)}}.
    \end{align}
    The proof amounts to showing that the second term is not positive. 
    There are two cases, depending on the weights indexed by $t_2$ and $N$. 
    
    \un{Case 1}: If $w_{t_2} = 1$ or $w_N = 1$, then $t_2$ and $N$ are never in the same part of the partition and lemma follows immediately.
    
    \un{Case 2}: If $w_{t_2} = w_{N} = \varepsilon$, we prepare notations for the analysis that follows: 
    \begin{enumerate}
        \item Set $S' : =S\smallsetminus\{t_2,N\}\cup\{\bullet\}$.
        \item Let ${k_{\bullet} := k_{t_2} + k_{N} - 1}$ and $K'$ be the sequence obtained from $K$ by deleting $k_{t_2}$ and $k_N$ and appending $k_{\bullet}$.
        \item $w' :=(1^{(n-m)},\varepsilon^{(m-1)})$ is the weight vector obtained from $w$ by deleting the weights indexed by $t_2$ and $N$, and appending the weight of the mark $\bullet$ as $\varepsilon$.
    \end{enumerate} 
    Then 
    \begin{align}
        &\sum_{\substack{P\in\mathcal{P}_w(S)\\t_2,N\,\textrm{in}\\
        \textrm{same part of}\,P}}(-1)^{|S|-\ell(P)}\binom{\sum_i K(P)_i}{K(P)_{1}, \ldots, K(P)_{\ell(P)}}\\
        &= -\sum_{\substack{P\in\mathcal{P}_{w'}(S')}}(-1)^{|S'|-\ell(P)}\binom{\sum_i K'(P)_i}{K'(P)_1, \ldots, K'(P)_{\ell(P)}} \\
        &\le 0
    \end{align}
    by \cref{lem:nonnegative}, thus proving the lemma.
\end{proof}

\bibliographystyle{alpha}
\bibliography{bibliography.bib}

\begin{thebibliography}{CHMR16}

\bibitem[AG08]{alexeev_guy_2008}
Valery Alexeev and G.~Michael Guy.
\newblock Moduli of weighted stable maps and their gravitational descendants.
\newblock {\em Journal of the Institute of Mathematics of Jussieu},
  7(3):425–456, 2008.

\bibitem[AR10]{Allermann_Rau_2010}
Lars Allermann and Johannes Rau.
\newblock First steps in tropical intersection theory.
\newblock {\em Mathematische Zeitschrift}, 264(3):633–670, Mar 2010.
\newblock arXiv: 0709.3705.

\bibitem[BC20]{Blankers_Cavalieri_2020_Wall}
Vance Blankers and Renzo Cavalieri.
\newblock {Wall-Crossings for Hassett Descendant Potentials}.
\newblock {\em International Mathematics Research Notices}, 2020.

\bibitem[Cey09]{ceyhan2009chow}
{\"O}zg{\"u}r Ceyhan.
\newblock Chow groups of the moduli spaces of weighted pointed stable curves of
  genus zero.
\newblock {\em Advances in Mathematics}, 221(6):1964--1978, 2009.

\bibitem[CGM20]{cavalieri2020tropical}
Renzo Cavalieri, Andreas Gross, and Hannah Markwig.
\newblock Tropical $\psi$-classes.
\newblock {\em arXiv preprint arXiv:2009.00586}, 2020.

\bibitem[CHMR14]{CHMR2014moduli}
Renzo Cavalieri, Simon Hampe, Hannah Markwig, and Dhruv Ranganathan.
\newblock Moduli spaces of rational weighted stable curves and tropical
  geometry.
\newblock {\em Forum of Mathematics, Sigma}, 4, 04 2014.

\bibitem[CHMR16]{Cavalieri_Hampe_Markwig_Ranganathan_2016}
Renzo Cavalieri, Simon Hampe, Hannah Markwig, and Dhruv Ranganathan.
\newblock Moduli spaces of rational weighted stable curves and tropical
  geometry.
\newblock {\em Forum of Mathematics, Sigma}, 4:e9, 2016.

\bibitem[FR13]{franccois2013diagonal}
Georges Fran{\c{c}}ois and Johannes Rau.
\newblock The diagonal of tropical matroid varieties and cycle intersections.
\newblock {\em Collectanea mathematica}, 64(2):185--210, 2013.

\bibitem[Fry19]{fry2019tropical}
Andy Fry.
\newblock Tropical moduli space of rational graphically stable curves.
\newblock {\em arXiv preprint arXiv:1910.00627}, 2019.

\bibitem[GM10]{gibney2010chow}
Angela Gibney and Diane Maclagan.
\newblock {Equations for Chow and Hilbert quotients}.
\newblock {\em Algebra \& Number Theory}, 4(7):855 -- 885, 2010.

\bibitem[Has03]{HASSETT2003316}
Brendan Hassett.
\newblock Moduli spaces of weighted pointed stable curves.
\newblock {\em Advances in Mathematics}, 173(2):316 -- 352, 2003.

\bibitem[Kat12]{Katz_2012}
Eric Katz.
\newblock Tropical intersection theory from toric varieties.
\newblock {\em Collectanea Mathematica}, 63(1):29–44, Jan 2012.

\bibitem[KKL21]{Kannan_2021}
Siddarth Kannan, Dagan Karp, and Shiyue Li.
\newblock Chow rings of heavy/light hassett spaces via tropical geometry.
\newblock {\em Journal of Combinatorial Theory, Series A}, 178:105348, Feb
  2021.

\bibitem[KM09]{Kerber_Markwig_2009}
Michael Kerber and Hannah Markwig.
\newblock {Intersecting Psi-classes on tropical $M_{0,n}$}.
\newblock {\em International Mathematics Research Notices}, 2009(2):221–240,
  2009.
\newblock arXiv: 0709.3953.

\bibitem[Kon92]{kontsevich1992intersection}
Maxim Kontsevich.
\newblock Intersection theory on the moduli space of curves and the matrix airy
  function.
\newblock {\em Communications in Mathematical Physics}, 147(1):1--23, 1992.

\bibitem[LM00]{losev2000new}
Andrey Losev and Yuri Manin.
\newblock New moduli spaces of pointed curves and pencils of flat connections.
\newblock {\em The Michigan Mathematical Journal}, 48(1):443--472, 2000.

\bibitem[{Mik}07]{zbMATH05543460}
Grigory {Mikhalkin}.
\newblock {Moduli spaces of rational tropical curves}.
\newblock In {\em Proceedings of the 13th G\"okova geometry-topology
  conference, G\"okova, Turkey, May 28--June 2, 2006.}, pages 39--51.
  Cambridge, MA: International Press, 2007.

\bibitem[Moo13]{moon2013log}
Han-Bom Moon.
\newblock Log canonical models for the moduli space of stable pointed rational
  curves.
\newblock {\em Proceedings of the American Mathematical Society},
  141(11):3771--3785, 2013.

\bibitem[Rau16]{rau2016intersections}
Johannes Rau.
\newblock {Intersections on tropical moduli spaces}.
\newblock {\em Rocky Mountain Journal of Mathematics}, 46(2):581 -- 662, 2016.

\bibitem[Tev07]{tevelev2007compactifications}
Jenia Tevelev.
\newblock Compactifications of subvarieties of tori.
\newblock {\em American Journal of Mathematics}, 129(4):1087--1104, 2007.

\bibitem[Uli15]{Ulirsch_2015}
Martin Ulirsch.
\newblock Tropical geometry of moduli spaces of weighted stable curves.
\newblock {\em Journal of the London Mathematical Society}, 92(2):427–450,
  Aug 2015.

\bibitem[Wit91]{witten1990two}
Edward Witten.
\newblock Two-dimensional gravity and intersection theory on moduli space.
\newblock {\em Surveys in differential geometry}, 1(1):243--310, 1991.

\end{thebibliography}

\end{document}